\documentclass[a4paper, UKenglish]{article}

\usepackage{babel}
\usepackage[utf8]{inputenc}
\usepackage[T1]{fontenc}

\usepackage{amsmath, amsthm, amssymb}
\usepackage{stmaryrd}
\usepackage{thmtools}
\usepackage{mathtools}
\usepackage{hyperref}
\usepackage[capitalise]{cleveref}
\usepackage{enumitem}
\usepackage{xspace}
\usepackage{todonotes}

\usepackage{tikz}

\Crefname{figure}{Figure}{Figures}

\declaretheorem[name=Theorem, numberwithin=section]{theorem}
\declaretheorem[name=Proposition, sibling=theorem]{proposition}
\declaretheorem[name=Lemma, sibling=theorem]{lemma}
\declaretheorem[name=Corollary, sibling=theorem]{corollary}

\declaretheorem[name=Claim, sibling=theorem]{claim}
\declaretheorem[name=Remark, style=remark, sibling=theorem]{remark}

\newenvironment{claimproof}[1][Proof of the Claim.]{\begin{proof}[#1]}{\end{proof}}

\newcommand{\Bb}{\mathcal{B}}
\newcommand{\Cc}{\mathcal{C}}

\newcommand{\Gc}{\mathcal{G}}
\newcommand{\Hc}{\mathcal{H}}
\newcommand{\Pc}{\mathcal{P}}

\newcommand{\Vc}{\mathcal{V}}
\newcommand{\Xc}{\mathcal{X}}
\newcommand{\G}{\mathbf{G}}
\newcommand{\Nn}{\mathbb{N}}
\newcommand{\Ff}{\mathbb{F}}

\newcommand{\Hbb}{\mathbb{H}}
\newcommand{\Sbb}{\mathbb{S}}

\newcommand{\tC}{\widetilde{C}}
\newcommand{\hHbb}{\widehat{\Hbb}}
\newcommand{\hH}{\widehat{H}}

\newcommand*\sg[1]{\{ #1 \}}
\newcommand{\tw}{\mathrm{tw}}
\newcommand{\diam}{\mathrm{diam}}
\newcommand{\pw}{\mathrm{pw}}

\newcommand{\bn}{\mathrm{bn}}
\newcommand{\hyp}{\mathsf{hypertorso}}
\newcommand{\bag}{\beta}
\newcommand{\lmap}{\lambda}
\newcommand{\rmap}{\rho}
\newcommand{\eqdef}{\coloneqq}
\newcommand{\abstr}[1]{\llbracket #1 \rrbracket}
\newcommand{\Abstr}{\mathbb{A}}
\newcommand{\Fcycles}{$\Ff_2$-cycles\xspace}
\newcommand{\Fcycle}{$\Ff_2$-cycle }

\renewcommand{\hat}[1]{\widehat{#1}}

\let\le\leqslant
\let\ge\geqslant
\let\leq\leqslant
\let\geq\geqslant

\let\emptyset\varnothing

\usepackage{authblk}

\title{Basis Number of Graphs Excluding Minors}

\author[1]{Colin Geniet\thanks{Supported by the Institute for Basic Science (IBS-R029-C1).}}
\author[2]{Ugo Giocanti\thanks{Supported by the National Science Center of Poland
under grant 2022/47/B/ST6/02837 within the OPUS 24 program.}}

\affil[1]{Discrete Mathematics Group, Institute for Basic Science, Daejeon, South Korea}
\affil[2]{Faculty of Mathematics and Computer Science, Jagiellonian University,
Kraków, Poland}

\date{}

\begin{document}
\maketitle

\begin{abstract}
  The \emph{basis number} of a graph~$G$ is the minimum~$k$ such that the cycle space of~$G$ is generated by a family of cycles using each edge at most~$k$ times.
  A classical result of Mac Lane states that planar graphs are exactly graphs with basis number at most~2,
  and more generally, graphs embeddable on a surface of bounded genus are known to have bounded basis number.
  Generalising this, we prove that graphs excluding a fixed minor~$H$ have bounded basis number.

  Our proof uses the Graph Minor Structure Theorem, which requires us to understand how basis number behaves in tree-decompositions. In particular, we prove that graphs of treewidth~$k$ have basis number bounded by some function of~$k$.
  We handle tree-decompositions using the proof framework developed by Bojańczyk and Pilipczuk in their proof of Courcelle's conjecture. 

  Combining our approach with independent results of Miraftab, Morin and Yuditsky (2025) on basis number and path-decompositions, one can moreover improve our upper bound to a polynomial one: there exists an absolute constant $c>0$ such that every $H$-minor free graph has basis number $O(|H|^c)$.
\end{abstract}
 
\section{Introduction}
Given two graphs $G=(V,E_G),H=(V,E_H)$ with the same vertex set $V$, one can define their \emph{sum} $G\oplus H$ as the graph with vertex set $V$, and whose edge set is the symmetric difference $E_G\Delta E_H$. The \emph{cycle space} of a (finite) graph $G$ is the set of all subgraphs of $G$ that can be expressed as a sum of
cycles of $G$. This corresponds exactly to the set of subgraphs of $G$ whose degrees are all even. 
When~$G$ is planar, a basic observation is that the set of \emph{facial} cycles forms a \emph{cycle basis} of $G$, i.e.\ a generating set of the cycle space.
In particular, this gives a cycle basis $\Bb$ such that each edge of $G$ appears it at most two elements of $\Bb$. Mac Lane proved that this gives in fact a characterisation of planar graphs.
\begin{theorem}[Mac Lane's planarity criterion, \cite{maclane1937planar}]\label{thm:maclane}
  A graph~$G$ is planar if and only if there exists a cycle basis~$\Bb$ of~$G$ for which each edge of $G$ appears in at most two elements of~$\Bb$.
\end{theorem}

More generally, we call \emph{edge-congestion} (or just \emph{congestion}) of a cycle basis~$\Bb$ the smallest integer~$k$ such that each edge of $G$ appears in at most~$k$ cycles of~$\Bb$.
Inspired by \cref{thm:maclane}, Schmeichel~\cite{schmeichel1981basisnumber} defined the \emph{basis number} of a graph~$G$, which we denote by~$\bn(G)$, as the minimum edge-congestion over all cycle bases of~$G$.
He proved that cliques have basis number only~3, yet there are graphs with arbitrarily large basis number,
and also that --- of particular interest to us --- graphs embedded in the orientable surface with genus~$g$ have basis number at most~$2g+2$.
Recently, this result was significantly improved by Lehner and Miraftab.
\begin{theorem}[{\cite[Theorem~4]{LM24}}]\label{thm:bn-surface}
  If $G$ is a graph embeddable in a surface (orientable or not) with genus at most $g$, then $\bn(G)=O(\log^2g)$.
\end{theorem}

A natural question, which also appeared in an early version of \cite{LM24}, is whether the aforementioned results in the planar and more generally in the bounded genus case generalise to proper minor-closed classes: does every graph excluding a fixed minor $H$ have basis number upper bounded by a constant that only depends of $H$? Our main result is a positive answer to this question.

\begin{restatable}{theorem}{minorbn}\label{thm: main-minor}
  There exists a function $f_{\ref{thm: main-minor}}:\mathbb N\to \mathbb N$ such that for any graph~$H$, any $H$-minor free graph~$G$ satisfies $\bn(G) \leq f_{\ref{thm: main-minor}}(|H|)$.
\end{restatable}

Our proof gives a double exponential bound $f_{\ref{thm: main-minor}}(t)=2^{2^{O(t^2)}}$.
As discussed in Section~\ref{sec: polynomial} below,
Miraftab, Morin, and Yuditsky~\cite{miraftab2026pathwidth} independently proved one of the key results we use to establish \cref{thm: main-minor},
with a significantly better bound than ours.
Using their result, the bound~$f_{\ref{thm: main-minor}}$ can be improved to a polynomial function.

A class of graphs is called \emph{monotone} if it is closed under taking subgraphs.
Let us observe that \cref{thm: main-minor} can be rephrased as a characterisation of bounded basis number among monotone classes of graphs.
\begin{corollary}\label{cor:monotone-bn}
  Let~$\Cc$ be a monotone class of graphs.
  Then~$\Cc$ has bounded basis number if and only if all graphs in $\Cc$ exclude some fixed graph~$H$ as a minor.
\end{corollary}
Indeed, \cref{thm: main-minor} proves the right-to-left implication.
Conversely, if~$\Cc$ is a monotone graph class with bounded basis number,
then its \emph{minor closure} (i.e.\ the class of all minors of graphs in $\Cc$) also has bounded basis number,
since contracting edges cannot increase the basis number (see e.g.\ \cite[Lemma~3.1]{bazargani2024planar}).
Thus, as there exist some graphs with arbitrarily large basis number \cite{schmeichel1981basisnumber}, the minor closure of~$\Cc$ cannot contain all graphs, i.e.~$\Cc$ excludes some fixed graph~$H$ as a minor.

Our approach to prove \cref{thm: main-minor} uses the Graph Minor Structure Theorem of Robertson and Seymour~\cite{RSXVI}, a cornerstone of the Graph Minors series.
Avoiding any formal statement for now (we refer to Section \ref{sec: minors} for definitions and for a more detailed statement), it shows that for any fixed~$H$,
all $H$-minor free graphs can be constructing thanks to the following procedure:
\begin{enumerate}
  \item taking graphs embedded on a fixed surface,
  \item glueing a bounded number of \emph{vortices of bounded width} inside the faces,
  \item adding a bounded number of \emph{apices} (i.e.\ vertices with arbitrary adjacencies to the rest of the graph),
  \item and finally combining the graphs obtained in the first three steps by doing some \emph{clique-sums of bounded order}.
\end{enumerate}
To prove \cref{thm: main-minor}, we thus want to show that each of these steps preserves bounded basis number.
\Cref{thm:bn-surface} handles the case of graphs embedded on a surface,
and it is simple to show that adding~$\ell$ apices increases the basis number by at most~$2\ell$ (see Lemma \ref{lem:bn-edits}).
Glueing vortices requires some more involved topological arguments, but our approach is conceptually quite natural, so we will not discuss it here.
What remains is thus to understand how basis number behaves with respect to clique-sums of bounded order,
or equivalently tree-decompositions of bounded adhesion (see \cref{sec: prelis} for definitions related to tree-decompositions).

First, we show that graphs with bounded treewidth have bounded basis-number.
\begin{restatable}{theorem}{twbn}\label{thm:bn-treewidth}
  There exists a function $f_{\ref{thm:bn-treewidth}}:\mathbb N\to \mathbb N$ such that for every $k\geq 0$, every graph with treewidth~$k$ has basis number at most~$f_{\ref{thm:bn-treewidth}}(k)$.
\end{restatable}
Here again, our proof gives the double exponential bound $f_{\ref{thm:bn-treewidth}}(k)=2^{2^{O(k^2)}}$,
but this can be improved to a polynomial thanks to~\cite{miraftab2026pathwidth}, see \cref{sec: polynomial}.

\Cref{thm:bn-treewidth} does not apply to tree-decompositions obtained through the Graph Minor Structure Theorem,
as the latter only have bounded adhesions, but not necessarily bounded width.
Our second main result is the following, whose statement is more technical, but which gives some sufficient conditions for a tree-decomposition with bounded adhesion to preserve the property of having bounded basis number. This will allow to handle the last part of our proof of Theorem \ref{thm: main-minor}.

\begin{restatable}{theorem}{tdbn}\label{thm: main-td}
 There exists a function $f_{\ref{thm: main-td}}: \mathbb N^2 \to \mathbb N$ such that the following holds.
 Let~$\Gc$ be a monotone class of graphs such that each graph in $\Gc$ has basis number at most~$b$,
 and let~$G$ be a graph with a tree-decomposition with adhesion at most~$k$ and whose torsos are all in~$\Gc$.
 Then
 \[ \bn(G) \le f_{\ref{thm: main-td}}(b,k). \]
\end{restatable}
Once again, our bound is a double exponential $f_{\ref{thm: main-td}}(b,k)=b\cdot 2^{2^{O(k^2)}}$, and improves to a polynomial thanks to~\cite{miraftab2026pathwidth}.

Our proofs of \cref{thm:bn-treewidth,thm: main-td} are based on techniques developed by Boja\'nczyk and Pilipczuk~\cite{mimi} in their proof of Courcelle's conjecture on MSO-definable tree-decompositions, see \cref{sec:proof-overview} for an overview.

\subsection{Polynomial bounds}
\label{sec: polynomial}
One of the steps of our proof consists in dealing separately with the case of graphs of bounded pathwidth, and more generally with path-decompositions of bounded adhesion.

Independently from our work, Miraftab, Morin, and Yuditsky proved that graphs of pathwidth~$k$ in fact have linear basis number.
\begin{theorem}[\cite{miraftab2026pathwidth}]\label{thm:pathwidth-linear}
  Any graph with pathwidth~$k$ has basis number at most~$4k$.
\end{theorem}
It turns out that the bounded pathwidth case is the only reason for the double-exponential~$2^{2^{O(k^2)}}$ bound in our proof of \cref{thm:bn-treewidth}.
Indeed, following the techniques of~\cite{mimi}, the first half of our proof shows that graphs of bounded pathwidth have bounded basis number,
using a Ramsey-like result from semigroup theory known as Simon's Factorisation Forests Theorem.
The double-exponential bound on basis number in \cref{thm:bn-treewidth} already appears at this step.
The second half of our proof --- generalising from bounded pathwidth to bounded treewidth --- only increases the basis number polynomially.

The work of Miraftab, Morin, and Yuditsky thus solves this bottleneck,
and combining \cref{thm:pathwidth-linear} with the second half of our proof of \cref{thm:bn-treewidth}, one can obtain a general polynomial bound on basis number in terms of the treewidth.
\begin{corollary}[\cite{miraftab2026pathwidth} and our results]\label{cor:bn-treewidth-polynomial}
  Any graph with treewidth~$k$ has basis number at most~$O(k^5)$.
\end{corollary}

The proof technique of Miraftab, Morin and Yuditsky generalises to graphs admitting some path-decompositions with bounded adhesion, whose parts have bounded basis number, with the following polynomial bound.
\begin{theorem}
 [\cite{miraftab2026pathwidth}]\label{thm:pd-linear}
 Let $b,k\in \mathbb N$ and let $G$ be a graph admitting a path-decomposition of adhesion $k$, in which each part has basis number at most $b$. Then 
 $$\bn(G)\leq b+O(k\log^2 k).$$
\end{theorem}

Again, 
combining Theorem \ref{thm:pd-linear} with our proof of Theorem \ref{thm: main-td}, one can obtain the improved polynomial bound $f_{\ref{thm: main-td}}(b,k)=O((b+k\log^2k) \cdot k^4)$.

Combining this with the recent polynomial bounds in the Graph Minor Structure Theorem obtained by Gorksy, Seweryn, and Wiederrecht~\cite{GSW25},
this in turn implies that one can obtain a polynomial bound for \cref{thm: main-minor}.

\begin{corollary}
 \label{coro: poly-minor}
  There exists a constant $c\leq 32210$ such that for every fixed graph $H$ and every $H$-minor free graph $G$, we have
  $$\bn(G)=O(|H|^c).$$
\end{corollary}

\subsection{Overview of the proof of Theorem \ref{thm: main-td}}\label{sec:proof-overview}
A very natural approach to prove \cref{thm:bn-treewidth,thm: main-td} is the following.
Consider a tree-decomposition~$(T,\beta)$ of~$G$, and for each bag~$\bag(t)$ of~$(T,\beta)$, choose a basis~$\Bb_t$ for~$G[\bag(t)]$ with small congestion.
One then wants to combine all the~$\Bb_t$ into a low congestion basis for~$G$.
It quickly becomes apparent that for this approach to work, any cycle~$C$ going through~$\bag(t)$ should project to some cycle inside~$\bag(t)$,
and to this end one should not consider the subgraph of $G$ induced by~$\bag(t)$, but rather its \emph{torso},
where one adds new edges corresponding to all connections made by paths in~$G$ outside~$G[\bag(t)]$.

Now imagine that~$C$ is a cycle in the basis~$\Bb_t$ which uses one of these edges~$e$ created in the torso of~$t$.
We want to use~$C$ as part of a basis of the cycle space of~$G$,
and this requires replacing~$e$ (which does not exist in~$G$), by some path in~$G$.
Naturally, this may happen for any edge~$e$ added in any torso of~$(T,\beta)$,
and one should care about the congestion of the family of paths used to replace all these edges.
This leads to the following result (see \cref{thm:bn-tree-decomp-paths} for a more precise statement),
whose proof uses elementary reasoning on cycle bases.
\begin{theorem}\label{thm:bn-tree-decomp-paths-simple}
  Let~$G$ be a graph, and $(T,\beta)$ a tree-decomposition of~$G$ with bounded adhesion, whose torsos have bounded basis number.
  For each edge~$e$ created in each torso of~$(T,\beta)$, pick a path in~$G$ joining the endpoints of~$e$,
  and assume that the family of all these paths has bounded congestion.
  Then~$G$ has bounded basis number.
\end{theorem}
Unfortunately, the assumptions from \cref{thm:bn-tree-decomp-paths-simple} fail even in some very simple cases:
in a long cycle, one can observe that no tree-decomposition of bounded width satisfies these hypotheses.
Nevertheless, this path system assumption turns out to be extremely similar to the conditions used by Bojańczyk and Pilipczuk in their proof of Courcelle's conjecture~\cite{mimi}.
Using the proof framework they developed, we are able to reduce \cref{thm:bn-treewidth,thm: main-td} to \cref{thm:bn-tree-decomp-paths-simple}.

This is done in two stages.
The first only considers path-decompositions, and applies Simon's Factorisation Forests Theorem
to obtain some regularity conditions on the connections between the different adhesions of the path-decomposition.
After removal of a bounded number of vertices, these conditions give the path system needed to apply \cref{thm:bn-tree-decomp-paths-simple}. Note that in view of the results from \cite{miraftab2026pathwidth}, this approach gives suboptimal bounds.

For the second stage, we use a technical result Bojańczyk and Pilipczuk, namely \cite[Lemma~5.9]{mimi}.
This result implies that any graph~$G$ of treewidth~$k$ has a tree-decomposition with bounded adhesions, whose torsos have bounded pathwidth, and with a path system as required by \cref{thm:bn-tree-decomp-paths-simple}. Together with the previous results, this immediately implies \cref{thm:bn-treewidth}. To prove \cref{thm: main-td}, we need to generalise \cite[Lemma~5.9]{mimi} to tree-decompositions with small adhesions but unbounded width.
Though quite technical, the ideas remain fundamentally the same as in~\cite{mimi}.

\paragraph{Structure of the paper}
We introduce in \cref{sec: prelis} all basic definitions and properties related to graphs, basis numbers and tree-decompositions.
Then, in \cref{sec: td1}, we prove \cref{thm:bn-tree-decomp-paths}, which gives some simple sufficient condition for a tree-decomposition of bounded adhesion to preserve basis number. This result will be extensively used in all subsequent sections.

\Cref{sec: pw} proves \cref{thm: main-td} in the special case of path-decompositions using Simon's Forest Factorisation Theorem,
and in particular shows that graphs of bounded path-width have bounded basis number.
Then, in \cref{sec:tw} we combine the results of the previous two sections with \cite[Lemma 5.9]{mimi} to prove \cref{thm:bn-treewidth}.
To further obtain \cref{thm: main-td}, we prove in \cref{sec: td2} a generalisation of \cite[Lemma 5.9]{mimi} to tree-decompositions of unbounded width.

We then explain in \cref{sec: minors} how to combine \cref{thm: main-td} and the Graph Minor Structure Theorem in order to derive a proof of \cref{thm: main-minor}.
This part still requires some technical work and the use of known results on graphs embedded in surfaces to handle vortices.
Finally, we give in \cref{sec: ccl} some examples of graph families with unbounded basis number, and conclude with some related questions and possible future directions.

\section{Preliminaries}
\label{sec: prelis}
We introduce in Section \ref{sec: graphs} some basic notions on graphs, and give then in Section \ref{sec: bn} the definition of the basis number of a graph, together with some basic first properties. Section \ref{sec: prelis-td} defines standard notions related to tree-decompositions, as well as some additional terminology from~\cite{mimi}.

\subsection{Graphs}
\label{sec: graphs}
In the whole paper, if $k$ is a non-negative integer, we let $[k]$ denote the set $\sg{1,\ldots,k}$ of integers.

Unless stated otherwise, all the graphs we will consider are finite, unoriented and without loops or multi-edges.
A \emph{path} $P$ in $G$ is a sequence $v_1\ldots v_k$ of pairwise distinct vertices such that for each $i\in [k-1]$, we have $v_iv_{i+1}\in E(G)$. Throughout the paper, a path $P$ will be identified with the subgraph of~$G$ consisting of the $v_i$'s and the edges $v_iv_{i+1}$. The \emph{length} of $P$ is its number $k-1$ of edges. If $v_1=x$ and $v_k=y$, then we say that $P$ is a \emph{$xy$-path}.
We let $d_G(\cdot, \cdot)$ denote the shortest path metric in $G$,  and let $\diam(G)$ denote the \emph{diameter} of the graph~$G$, that is, the largest possible distance between two vertices. 

For every graph $G$, and every subset $X\subseteq V(G)$ of vertices, we let $G[X]\eqdef (X,E(G)\cap \binom{X}{2})$ denote the subgraph of $G$ induced by $X$, and we set $G-X\eqdef G[V(G)\setminus X]$. Similarly, for every $F\subseteq E(G)$, we set $G-F\eqdef  (V(G),E(G)\setminus F)$, and for simplicity, for every $v\in V(G), e\in E(G)$ we write $G-v\eqdef G-\sg{v}$ and $G-e\eqdef G-\sg{e}$.
A \emph{spanning subgraph} of~$G$ is a subgraph containing all vertices of~$G$.

For every two graphs $G_1, G_2$, whose vertex and edge sets possibly intersect, their union is the graph $G_1\cup G_2\eqdef (V(G_1)\cup V(G_2), E(G_1)\cup E(G_2))$,
and similarly the intersection is $G_1 \cap G_2 \eqdef (V(G_1) \cap V(G_2), E(G_1) \cap E(G_2))$.

\begin{remark}
\label{rem: connecting-forest}
Let $G$ be a connected graph. For every family $H_1,\ldots, H_k$ of pairwise vertex-disjoint subgraphs of $G$, there exists some subforest $T$ of $G$ such that the graph $T\cup H_1\cup \cdots \cup H_k$ is connected, and such that the cycles from $G$ are exactly the ones of $H_1\cup \cdots \cup H_k$. We call such a forest $T$ a \emph{connecting forest} for $H_1,\ldots,H_k$ in~$G$.
\end{remark}

The \emph{square} of a graph $G$ is the graph $G^2$ with vertex set $V(G)$ and which contains an edge between every two vertices at distance at most $2$ in $G$. Note that we always have $\diam(G)\leq 2\cdot \diam(G^2)$.

If $G$ is a graph, for every $u\in V(G)$, we let $N_G[u]\eqdef \sg{v\in V(G): uv\in E(G)}\cup \sg{u}$ denote the \emph{closed neighbourhood} of $u$ in $G$, and for every subset $X\subseteq V(G)$, we set $N_G[X]\eqdef \bigcup_{u\in X}N_G[u]$. 
A \emph{dominating set} in a graph $G$ is a subset $D\subseteq V(G)$ such that $V(G)=N_G[D]$.

\begin{remark}
 \label{rem: dominating-diameter}
 Let $G$ be a connected graph and $D\subseteq V(G)$ be a dominating set of $G$. Then $\diam(G)\leq 3|D|-1$.
\end{remark}
\begin{proof}
Let $x,y\in V(G)$ and consider a shortest $xy$-path $P$ in $G$. Then note that for every $u\in D$, $P$ must intersect $N_G[u]$ in at most $3$ vertices. This implies that~$P$ has at most $3|D|$ vertices.
\end{proof}

In a graph~$G$, a \emph{separation} is a pair~$(X,Y)$ of subsets of vertices such that $V(G) = X \cup Y$,
and there are no edges joining~$X \setminus Y$ and~$Y \setminus X$, i.e.\ these two sets are separated by~$X \cap Y$.
The \emph{order} of the separation is~$|X \cap Y|$.

\subsection{Basis number}
\label{sec: bn}
Given a (finite) graph~$G$, we are interested in subsets of~$E(G)$.
They form an $\Ff_2$-vector space, where the sum operation is the symmetric difference, denoted by~$X \oplus Y$.
We frequently identify subgraphs of~$G$ with their set of edges.
We call \emph{$\Ff_2$-cycle} a subgraph of $G$ whose vertices all have even degree,
and reserve the name \emph{cycle} for the graph theoretic notion, i.e.\ connected subgraphs whose vertices all have degree~2.
Any $\Ff_2$-cycle can be obtained as disjoint union of cycles.
The set of all $\Ff_2$-cycles is a vector subspace of~$\Ff_2^{E(G)}$ called the \emph{cycle space} of~$G$.
A basis for the cycle space is called a \emph{cycle basis}.
A well-known and easy fact is that the cycle space of a finite connected graph $G$ has dimension exactly $|E(G)|-|V(G)|+1$.

The (edge-)congestion of a family~$\Xc = \{X_1,\dots,X_n\}$ of subgraphs of~$G$ is
\[ \max_{e \in E(G)} |\{i : e \in X_i\}|. \]
The \emph{basis number} of~$G$, denoted~$\bn(G)$, is the minimum edge-congestion of a cycle basis of~$G$.
One can always find a cycle basis with minimum congestion consisting only of cycles.

The next result is a fundamental property of basis number (see for example \cite[Lemma 3.8]{bazargani2024planar} in the case where $G_X\cap G_Y$ is spanning).
\begin{lemma}\label{lem:bn-union}
  Let~$G$ be a graph with two subgraphs~$G_X,G_Y$ such that $G = G_X \cup G_Y$, and $G_X \cap G_Y$ is connected.
  If~$\Bb_X,\Bb_Y$ are cycle bases of~$G_X$ and~$G_Y$ respectively, then~$\Bb_X \cup \Bb_Y$ generates the cycle space of~$G$.
  In particular,
  \[ \bn(G) \le \bn(G_X) + \bn(G_Y). \]
\end{lemma}
\begin{proof}
  Take~$C$ an $\Ff_2$-cycle in~$G$, and partition it into $C = C_X \oplus C_Y$ such that $C_X \subseteq G_X$ and $C_Y \subseteq G_Y$ (these are not necessarily $\Ff_2$-cycles).
  Let~$U$ be the set of vertices with odd degree in~$C_X$.
  Clearly~$U$ is also equal to the set of odd-degree vertices in~$C_Y$, hence $U \subseteq V(G_X \cap G_Y)$. Moreover, $|U|$ must be even.

  Since~$G_X \cap G_Y$ is connected, there is some subgraph $H \subseteq G_X \cap G_Y$
  such that~$U$ is exactly the set of odd-degree vertices of~$H$:
  if the vertices of~$U$ are~$u_1,\dots,u_{2n}$ ordered arbitrarily,
  then pick an $u_{2i-1}u_{2i}$-path~$P_i$ in $G_X \cap G_Y$ for each $i \in [n]$, and define $H = \bigoplus_{i=1}^n P_i$.
  Now $C'_X \eqdef C_X \oplus H$ and $C'_Y \eqdef C_Y \oplus H$ are \Fcycles in~$G_X$ and~$G_Y$ respectively, and $C = C'_X \oplus C'_Y$, proving the result.
\end{proof}

In the remainder of \cref{sec: bn}, we establish a number of consequences of \cref{lem:bn-union},
starting with small graph modifications that preserve basis number.

\begin{lemma}\label{lem:bn-edits}
  For any graph~$G$, we have
  \begin{align}
    \label{eq:bn-vertex-add} \bn(G) & \le \bn(G-v) + 2 && \text{for any vertex~$v$} \\
    \label{eq:bn-edge-add} \bn(G) & \le \bn(G-A) + O(\log^2 |A|) && \text{for any subset~$A$ of edges} \\
    \label{eq:bn-edge-del} \bn(G-A) & \le (|A|+1)\cdot \bn(G) && \text{for any subset~$A$ of edges}
  \end{align}
\end{lemma}
\begin{proof}
\begin{enumerate}[label=(\arabic*)]
 \item Pick~$F$ a spanning forest of~$G$, and consider the following two subgraphs.
   First set $G_1 \eqdef  F \cup (G-v)$. This consists exactly of~$G-v$, plus for each connected component~$X$ of~$G-v$, a single edge from~$v$ to~$X$.
   These additional edges do not create any new cycle, thus $\bn(G_1) = \bn(G-v)$.
   Next define~$G_2$ by adding to~$F$ all edges incident to~$v$. This is a planar graph, hence $\bn(G_2) \le 2$.
   Clearly $G = G_1 \cup G_2$, while $G_1 \cap G_2 = F$ is connected. \Cref{lem:bn-union} then gives
   \[ \bn(G) \le \bn(G_1) + \bn(G_2) \le \bn(G-v) + 2. \]

 \item Without loss of generality, we may assume that~$G$ is connected.
   Furthermore, if~$G-A$ is not connected, then there is an edge $e \in A$ joining two components of~$G-A$.
   For $A' \eqdef A \setminus \{e\}$, we then have $\bn(G-A) = \bn(G - A')$.
   By induction on the size of~$A$, we may assume $\bn(G) \le \bn(G-A') + O(\log^2|A'|)$, which implies the desired bound.

   We thus assume that~$G-A$ is connected too.
   Pick~$F$ a spanning tree of~$G-A$ and define $H \eqdef  F \cup G[A]$. Then~$H$ can be embedded on a surface of genus at most~$|A|$, by drawing~$F$ on the sphere and adding a handle to draw each edge of $A$.
   By \cref{thm:bn-surface}, this gives~$\bn(H) = O(\log^2 |A|)$.
   Now $G-A$ and~$H$ are two subgraphs whose union is~$G$, and whose intersection~$F$ is connected.
   By \cref{lem:bn-union}, we obtain $\bn(G) \le \bn(G-A) + O(\log^2 |A|)$.

 \item Once again, without loss of generality we may assume that $G$ is connected,
   and if~$G-A$ is not connected, then picking $e \in A$ between two components and defining $A' \eqdef A \setminus \{e\}$
   yields $\bn(G-A) = \bn(G-A')$, allowing to conclude by induction on the size of~$A$.

  We may thus assume now that each edge $e\in A$ is not a cutedge of $G-(A\setminus \sg{e})$, i.e.\ there exists a path $P_e$ in $G-A$ connecting the two endpoints of $e$. For each \Fcycle $C$ of $G$, we let $\tC$ denote the subgraph of~$G-A$ obtained from $C$ after replacing every edge $e\in F\cap E(C)$ by the path~$P_e$. More formally, $\tC=C\oplus (\bigoplus_{e\in A\cap E(C)} (P_e\oplus \sg{e}))$. Note that as every path $P_e$ avoids $A$, $\tC$ is a \Fcycle in $G-A$, and that for every two \Fcycles $C,C'$ in $G$, the mapping $f:C\mapsto \tC$ satisfies  $f(C\oplus C')=f(C)\oplus f(C')$.
 Let $\Bb$ denote a cycle basis of $G$ with congestion $\bn(G)$. 
 We set $\Bb'\eqdef \sg{\tC: C\in \Bb}$, and observe that the equality $f(C\oplus C')=f(C)\oplus f(C')$ implies that
 $\Bb'$ generates the cycle space of $G-A$.
 
 We claim that $\Bb'$ has congestion at most $(|A|+1)\cdot \bn(G)$. Indeed, note that for every edge $e'\in E(G-A)$, 
 if $\tC$ is a \Fcycle in $\Bb'$ containing~$e'$, then either $C$ also contains~$e'$, or $C$ contains some edge $e\in A$, and $e'$ is on the path $P_e$. 
 In particular, as $\Bb$ has edge-congestion $\bn(G)$, this implies that $e'$ is contained in at most $(|A|+1)\cdot \bn(G)$ \Fcycles of $\Bb'$, as desired. \qedhere
\end{enumerate}
\end{proof}
Note that \cite[Open Problem 3.7]{bazargani2024planar} asks whether~$(3)$ holds with only an additive increase to the basis number.
Unlike the previous three operations, deleting a vertex can increase the basis number arbitrarily.
\begin{remark}
\label{rem: G-v}
  For any~$k \in \Nn$, there exists a graph~$G$ with $\bn(G) \le 3$, and a vertex $v \in V(G)$ such that $\bn(G-v) \ge k$.
\end{remark}
\begin{proof}
  Take a cubic graph~$H$ with basis number at least~$k$ (it suffices to pick~$G$ with sufficiently large girth, see \cref{prop: expanders}).
  Define~$G$ by adding to~$H$ a vertex~$v$ connected to all of~$V(H)$.
  Then the set of all triangles of the form~$vab$ for $ab \in E(H)$ generates the cycle space of~$G$,
  and since~$H$ is cubic, each edge of the form~$av$ appears in at most~3 such triangles, hence this family has congestion~3.
\end{proof}

\Cref{lem:bn-union} can also be phrased in terms of separators, which proves to be more convenient in the context of tree-decompositions.
First, it directly implies the following for connected separators.
\begin{lemma}\label{lem:bn-connected-separator}
  Consider a graph~$G$ and a separation~$(X,Y)$ with $G[X \cap Y]$ connected.
  If~$\Bb_X,\Bb_Y$ are cycle bases of~$G[X]$ and~$G[Y]$ respectively, then~$\Bb_X \cup \Bb_Y$ generates the cycle space of~$G$.
  In particular,
  \[ \bn(G) \le \bn(G[X]) + \bn(G[Y]). \]
\end{lemma}
Indeed, \cref{lem:bn-connected-separator} is a special case of  \cref{lem:bn-union} when taking $G_X = G[X]$ and $G_Y = G[Y]$.
For non-connected separators, next lemma shows that one can still obtain a bound depending on the separator size.
\begin{lemma}[\cite{miraftab2026pathwidth}]\label{lem:bn-small-separator}
  For any graph~$G$ and separation $(X,Y)$ of order $|X \cap Y| = k$,
  \[ \bn(G) \le \bn(G[X]) + \bn(G[Y]) + O(\log^2 k).\]
\end{lemma}
We thank Pat Morin for suggesting the use of \cref{thm:bn-surface}, improving a bound we had in an earlier version of \cref{lem:bn-small-separator} (see also Lemmas~8 and~9 in~\cite{miraftab2026pathwidth}).

\begin{proof}
  Without loss of generality, we may assume that $G$ is connected.
  Take~$F$ a spanning tree of~$G$, and define $G_X \eqdef G[X] \cup F$ and $G_Y \eqdef G[Y] \cup F$.
  We claim that $\bn(G_X) \le \bn(G[X]) + O(\log^2 k)$ and similarly for~$G_Y$, which implies the result by \cref{lem:bn-union}.

  Let~$A$ be a set of~$k-1$ edges such that the~$k$ vertices of~$X \cap Y$ are all in different components of~$F-A$.
  Then $G[X] \cup (F-A)$ consists of~$G[X]$ plus~$k$ disjoint trees branching out from each of the vertices of~$X \cap Y$.
  These trees do not create any additional cycle, hence $\bn(G[X] \cup (F-A)) = \bn(G[X])$.
  To obtain~$G_X$, it suffices to add back the missing edges of~$A$, which by \cref{lem:bn-edits} case~\eqref{eq:bn-edge-add} increases the basis number by at most~$O(\log^2 k)$,
  proving $\bn(G_X) \le  \bn(G[X]) + O(\log^2 k)$ as desired.
\end{proof}

Finally, we will need a variant of \cref{lem:bn-union} with several connected components.
\begin{lemma}\label{lem:bn-almost-connected-separator}
  Consider a graph~$G$ and two subgraphs~$G_X,G_Y$ such that $G = G_X \cup G_Y$,
  and such that for each connected component~$U$ of~$G_X$, the graph $G_X[U] \cap G_Y$ is connected.
  If~$\Bb_X,\Bb_Y$ are cycle bases of~$G_X$ and~$G_Y$ respectively, then~$\Bb_X \cup \Bb_Y$ generates the cycle space of~$G$.
  In particular,
  \[ \bn(G) \le \bn(G_X) + \bn(G_Y). \]
\end{lemma}
\begin{proof}
  Assume without loss of generality that all the graphs of
  $\Bb_X$ are connected.
  Consider $G_X^1,\dots,G_X^m$ the connected components of~$G_X$. For each $i\in [m]$, let $\Bb_X^i$ denote the set of cycles of $\Bb_X$ included in the component $G_X^i$.
  Clearly the basis~$\Bb_X$ is the disjoint union of cycle bases~$\Bb_X^i$ of~$G_X^i$ for each $i \in [m]$.
  Since $G_X^1,\dots,G_X^m$ are disjoint subgraphs, we have for any~$i \in [m]$
  \[ G_X^i \cap \Big(G_Y \cup \bigcup_{j<i} G_X^j\Big) = G_X^i \cap G_Y, \]
  and the right hand side is connected by assumption.
  An immediate induction using \cref{lem:bn-union} then gives that $\Bb_Y \cup \left(\bigcup_{j \le i} \Bb_X^i\right)$ generates the cycle space of $G_Y \cup \left(\bigcup_{j<i} G_X^j\right)$, proving the result when $i=m$.
\end{proof}

\subsection{Tree-decompositions}
\label{sec: prelis-td}
For tree-decompositions, we will use some of the terminology from \cite{mimi}.
A \emph{tree-decomposition} of a graph~$G$ is a pair~$(T,\beta)$, where $T$ is a rooted tree and
$\beta: V(T)\to 2^{V(G)}$ is a mapping that associates
for each \emph{node}~$t \in V(T)$, some subset~$\bag(t)$ of $V(G)$ called the \emph{bag} of~$t$, so that the following properties hold:
\begin{enumerate}[label=($\mathcal T$\arabic*)]
  \item\label{item: td1} $V(G)=\bigcup_{t\in V(T)}\beta(t)$,
  \item\label{item: td2} for each edge~$e=uv$ in~$G$, $\sg{u,v}$ is contained in some bag of~$T$, and
  \item\label{item: td3} for each vertex~$v$ of~$G$, the set of nodes of~$T$ whose bags contain~$v$ induces a connected subtree of~$T$.
\end{enumerate}

\begin{remark}\label{rmk:td-connected-subgraph}
  The three conditions imply that
  for any non-empty connected subgraph $H \subseteq G$, the subtree $T_X \eqdef \{t \in V(T) : \bag(t) \cap V(H) \neq \emptyset\}$ of $T$ induced by the nodes of $T$ whose bags contain some vertex of~$H$ induces a non-empty connected subtree of~$T$.
\end{remark}

The \emph{width} of a tree-decomposition is the maximum size of its bags, minus~1, and
the \emph{treewidth} of~$G$, denoted~$\tw(G)$ is the minimum width of a tree-\break decomposition of~$G$.

For a node~$t \in V(T)$, the \emph{adhesion} of~$t$ is the intersection of the bag of~$t$ with the bag of its parent (recall that $T$ is rooted).
The adhesion of the root of~$T$ is defined as being the empty set.
We say that the tree-decomposition \emph{has adhesion at most~$k$} if the adhesion of each node has size at most~$k$.

The \emph{margin} of a node~$t$ consists of its bag minus its adhesion.
The \emph{component} of~$t$ is the union of the margins of all descendants of~$t$ in~$T$ (including~$t$ itself). Equivalently, the component of~$t$ is the union of the bags of all descendants of~$t$, minus the adhesion of~$t$.
The \emph{part} of $t$ is the induced subgraph $G[\beta(t)]$, and the \emph{torso} of~$t$ is the graph obtained from the part of~$t$ after cliquifying each adhesion,
that is, adding edges~$uv$ for each pair~$u,v$ contained in the adhesion of either~$t$ or one of its children.
The \emph{marginal graph} of~$t$ is the torso of~$t$ with the adhesion vertices of~$t$ removed.
In other words, it is the subgraph induced by the margin of~$t$, with the adhesions of children of~$t$ cliquified.

When there is ambiguity regarding the tree-decomposition~$(T,\beta)$ considered,
we will also talk about $(T,\beta)$-bag, $(T,\beta)$-adhesion, etc., or simply of $T$-bag, $T$-adhesion, etc., when the context is clear.

\begin{remark}\label{rmk:torso-adh}
  Consider~$(T,\beta)$ a tree-decomposition, $t$ a node of~$T$, and~$u,v \in \bag(t)$.
  If the pair~$\{u,v\}$ is contained in \emph{any} adhesion of~$(T,\beta)$, then~$uv$ is an edge in the torso of~$t$.
  Indeed, in this case, $\{u,v\}$ is both in~$\bag(t)$ and some different bag, which implies it is in the adhesion of~$t$ or one of its children.
\end{remark}

A \emph{path-decomposition} is a special case of tree-decomposition where the tree~$T$ is a path.
If $(P, \beta)$ is a path-decomposition, with $P=v_1,\ldots,v_p$, then we respectively call $\beta(v_1)$ and $\beta(v_p)$ the \emph{leftmost} and \emph{rightmost} bags of $(P,\beta)$.
The pathwidth of~$G$, denoted~$\pw(G)$, is the minimum width of a path-decomposition.

A tree-decomposition~$(T,\beta)$ is called \emph{sane} if the following holds at every node~$t \in V(T)$:
\begin{enumerate}
  \item \label{sane:margin} the margin of~$t$ is non-empty,
  \item \label{sane:component} the component of~$t$ is connected, and
  \item \label{sane:adhesion} each vertex in the adhesion of~$t$ has a neighbour in the component of~$t$.
\end{enumerate}
\begin{lemma}[{\cite[Lemma~2.8]{mimi}}]
  \label{lem:sane}
  For any tree-decomposition~$(T,\beta)$ of a connected graph~$G$, there is a sane tree-decomposition~$(T',\beta')$ of~$G$
  such that each bag of~$(T',\beta')$ is contained in some bag of~$(T,\beta)$.
\end{lemma}
In particular, \cref{lem:sane} shows that one can always find a sane tree-\break decompositions of minimal width.

\begin{remark}
 \label{rem: sane-adhesion}
 In fact, the proof from \cite[Lemma~2.8]{mimi} can be reproduced exactly as it is  to show that we can moreover assume that all the adhesion sets of the obtained sane tree-decomposition $(T',\beta')$ are subsets of adhesion sets of the initial tree-decomposition $(T,\beta)$.
 For technical reasons, this additional property will turn out to be useful later.
 In particular, note that this implies that adhesions in~$(T', \beta')$ are no larger than the ones of~$(T,\beta)$,
 and by \cref{rmk:torso-adh} that torsos of~$(T',\beta')$ are subgraphs of torsos of $(T,\beta)$.
\end{remark}

\section{Adhesions captured by a path system}
\label{sec: td1}
This section proves Theorem \ref{thm:bn-tree-decomp-paths}, which gives some sufficient conditions on a tree-decomposition~$(T,\beta)$ ensuring that
a bound on the basis number of each torso of~$(T,\beta)$ yields a bound on the basis number of the whole graph.
Note that the conditions of Theorem \ref{thm:bn-tree-decomp-paths} cannot always be satisfied: for example, no tree-decomposition of bounded width of a long cycle satisfies these properties.
In \cref{sec: pw,sec:tw,sec: td2}, we will adapt the machinery of~\cite{mimi} to reduce \cref{thm: main-td} to cases where \cref{thm:bn-tree-decomp-paths} can be applied.

Let~$(T,\beta)$ be a tree-decomposition of a graph~$G$.
We say that a multi-set of paths~$\Pc$ of~$G$ \emph{captures the adhesions of~$(T,\beta)$} if the following holds:
$\Pc$ consists of paths~$P_{t,u,v}$, where~$t$ ranges over all nodes of~$(T,\beta)$, $u,v$ ranges over all pairs of vertices in the adhesion of~$t$, and~$P_{t,u,v}$ is a $uv$-path in~$G$.
The index~$t$ in~$P_{t,u,v}$ is only here to indicate that~$\Pc$ must contain as many $uv$-paths as there are nodes~$t$ whose adhesion contains~$u,v$; there is no requirement relating~$t$ and~$P_{t,u,v}$, the latter can be any $uv$-path of $G$.
It is also possible for~$P_{t,u,v}$ and~$P_{t',u,v}$ to be the same path for~$t \neq t'$. 

As long as~$G$ is connected, finding a family of paths that captures the adhesions of~$(T,\beta)$ is trivial.
We will be interested in finding such families with bounded congestion.
Here, the congestion of the family~$\Pc$ is
\[ \max_{e \in E(G)} |\{(t,u,v) : e\in P_{t,u,v} \}|, \]
meaning that we count paths with multiplicity:
when two paths~$P_{t,u,v},P_{t',u,v}$ in~$\Pc$ are equal, they both contribute to the congestion of~$\Pc$.
\begin{theorem}\label{thm:bn-tree-decomp-paths}
  Let~$G$ be a graph and~$(T,\beta)$ a tree-decomposition of~$G$ such that
  \begin{enumerate}
    \item for each~$t \in V(T)$, the torso of~$t$ has basis number at most~$b$, and
    \item \label{item:path-system} there is a family of paths~$\Pc$ of $G$ capturing the adhesions of~$T$ with edge congestion at most~$c$.
  \end{enumerate}
  Then $\bn(G) \le (2c+1)(b+1)$.
\end{theorem}

\begin{proof}
  For each node~$t$ of~$T$, consider a cycle basis~$\Bb_t$ for the torso of~$t$ with congestion at most~$b$.
  We construct a generating set~$\Bb$ of the cycle space of~$G$ as follows. We recall that the tree $T$ is rooted.
  \begin{enumerate}
    \item \label{item:path-switching} For each pair of vertices~$u,v$, consider all the nodes (when they exist) $t_1,\dots,t_k\in V(T)$ whose adhesion contains~$\{u,v\}$, ordered arbitrarily.
      Then we add to~$\Bb$ the \Fcycles $P_{t_i,u,v} \oplus P_{t_{i+1},u,v}$ for each~$i<k$.
      If~$uv$ is an edge in~$G$, then we also add the \Fcycle $\{uv\} \oplus P_{t_1,u,v}$.
    \item \label{item:fixed-cycles} For each node~$t \in V(T)$ and each \Fcycle $C \in \Bb_t$,
      we add to~$\Bb$ one \Fcycle $\tC$ of $G$ obtained from~$C$ as follows.
      First, for each $t\in V(T)$ and each edge $e=uv$ in the adhesion of $t$, we set $C_{e,t}\eqdef e\oplus P_{t,u,v}$. Note that~$e$ is an edge of the torso of $t$, but not necessarily of $G$, hence $C_{e,t}$ is not necessarily an \Fcycle in $G$.
      Now, if~$e$ is an edge of~$C$ which is contained in the adhesion of some node of $T$, then~$e$ is in the adhesion of either~$t$ or of some child~$t'$ of~$t$.
      Choose~$t_e$ to be one such node, choosing arbitrarily when several options are available.
      Finally, $\tC$ is defined as the sum of~$C$ together with all $C_{e,t_e}$, for each edge~$e$ of~$C$ contained in some adhesion of~$(T,\beta)$.
      We insist that this sum ranges over all edges~$e$ contained in some adhesion, not just the ones that are present in the torso but missing in~$G$.
      Each edge $e\in E(C)\setminus E(G)$ appears only in~$C$ and~$C_{e,t_e}$ in this sum, hence~$\tC$ is an \Fcycle of~$G$.
  \end{enumerate}
  We say that the \Fcycles of~$\Bb$ added in these two steps are respectively of \emph{type}~1 and~2, and for each $i\in \sg{1,2}$, we let $\Bb^{i}$ denote the set of \Fcycles of type $i$ in~$\Bb$, so that $\Bb=\Bb^1\cup \Bb^2$.

  \paragraph{Congestion}
  Let us first bound the congestion of~$\Bb$, considering the two types of \Fcycles separately. 
  Observe that for every fixed triple $(t,u,v)$ such that $t\in V(T)$ and $u,v$ belong to the adhesion of~$t$, the path $P_{t,u,v}$ is only used in the definition of at most two \Fcycles of type~1. Moreover, 
  each edge $e=uv\in E(G)$ included in some adhesion of $(T,\beta)$, is used in the definition of exactly one of the \Fcycles of type~1 of the form $e\oplus P_{t,u,v}$ (namely when $t=t_1$). As $\Pc$ has congestion at most $c$, 
  this implies that the congestion of $\Bb^1$ is at most~$2c+1$.

  Now, consider an edge~$e$ of $G$, and assume that it appears in an \Fcycle $\tC$ of type~2 constructed from some \Fcycle $C \in \Bb_t$, for some $t\in V(T)$.
  This can occur for two different reasons:
  \begin{enumerate}[label = (\roman*)]
    \item\label{item:congestion-core} 
      Either~$e$ is an edge of~$C$ that appears in none of the \Fcycles $C_{e', t_{e'}}$ for any edge $e'$ of $C$ belonging to an adhesion of $(T,\beta)$.
      Then we say that~$e$ is a \emph{core edge of~$\tC$}.
      In particular, it implies that~$e$ does not belong to any adhesion of~$(T,\beta)$, and that~$\beta(t)$ is the only bag of $(T,\beta)$ that contains~$e$.

    Now if a different cycle~$\tC' \in \Bb$ coming from $C' \in \Bb_{t'}$ also uses~$e$ as a core edge,
    then the above implies $t = t'$, and the cycles $C,C'$ in~$\Bb_t$ both use~$e$.
    As $\Bb_t$ has edge-congestion at most~$b$, it follows that at most~$b$ \Fcycles $\tC$ of type~2 contain~$e$ as a core edge.

  \item\label{item:congestion-substitute} Otherwise~$e$ is used in~$\tC$ as part of some \Fcycle $C_{e', t_{e'}}$, and we say that~$e$ is a \emph{substitute edge in~$\tC$}.
      In this case, $e$ must also appear in some path~$P_{t',u,v}$ which is involved in the definition of $\tC$.
    Since the family~$\Pc$ has congestion~$c$ it suffices to prove that each path~$P_{t',u,v}$ is used in the definition of at most~$2b$ \Fcycles of type~2 to obtain that~$e$ is used at most~$2cb$ times as substitute edge.
  \end{enumerate}
  Thus we only need to show that the number of times that a path~$P_{t',u,v}$ is used in the definition of some \Fcycle $\tC$ of type~2 is at most~$2b$.
  If $\tC \in \Bb$ was obtained from~$C \in \Bb_t$, then the construction of~$\tC$ guarantees
  that~$\tC$ only uses paths~$P_{t',u,v}$ where~$t'$ is either~$t$ itself or one of its children.
  Thus conversely, a given path~$P_{t',u,v}$ is only used for \Fcycles of type~2 coming from~$\Bb_{t'}$ or~$\Bb_t$, where~$t$ is the parent of~$t'$ in~$T$.
  That is, $P_{t',u,v}$ is only used to replace an occurrence of~$uv$ in some \Fcycle of either~$\Bb_t$ or~$\Bb_{t'}$.
  Thus~$P_{t',u,v}$ is used in the definition of at most~$2b$ \Fcycles of type~2.

  To summarize, each edge $e\in E(G)$ is contained in at most~$b$ cycles of~$\Bb^2$ as a core edge, and at most~$2bc$ cycles as a substitute edge,
  proving that~$\Bb^2$ has congestion at most $2bc+b$.
  Recall also that~$\Bb^1$ has congestion at most~$2c+1$.
  Combining the two, we find that $\Bb = \Bb^1 \cup \Bb^2$ has congestion at most $2bc+2c+b+1 = (2c+1)(b+1)$.

  \paragraph{Generating cycles}
  Let us now prove that~$\Bb$ indeed generates the cycle space of~$G$. The arguments we present are standard, and similar to the ones used in~\cite{HamannPlanar} and \cite[Section 5.3]{EGL23}.

  As an intermediate step, consider the graph~$G^+$ obtained by glueing all the torsos of~$T$,
  i.e.~$G^+$ has the same vertices as~$G$, and $xy \in E(G^+)$ whenever~$xy$ is an edge in the torso of some~$t \in T$.
  Further, define $\Bb^+ \eqdef \bigcup_{t \in T} \Bb_t$, which is a family of cycles in~$G^+$.

  \begin{claim}
  \label{clm: gen-torso}
    The family~$\Bb^+$ generates the cycle space of~$G^+$.
  \end{claim}
  \begin{claimproof}
    The graph~$G^+$ is obtained from the collection of torsos of~$T$ by iteratively glueing them along their adhesions, which are cliques.
    \cref{lem:bn-connected-separator} shows that when glueing two graphs~$G_1,G_2$ into~$G$ along a clique,
    one may generate the cycle space of~$G$ by taking the union of cycle bases of~$G_1$ and~$G_2$.
    Repetitively applying this lemma to the construction of~$G^+$,
    we then obtain that the cycle space of~$G^+$ is generated by the union of cycle bases of all torsos, that is~$\Bb^+$.
  \end{claimproof}

  Consider the linear projection~$f : \Ff_2^{E(G^+)} \to \Ff_2^{E(G)}$ which maps any edge of~$E(G)$ to itself,
  and maps an edge $uv \in E(G^+) \setminus E(G)$ to~$P_{t,u,v}$ for some arbitrary choice of a node~$t\in V(T)$ whose adhesion contains~$uv$.
  Clearly, $f$ maps \Fcycles to \Fcycles, and its restriction to~$\Ff_2^{E(G)}$ is the identity map,
  hence~$f$ maps the cycle space of~$G^+$ to the cycle space of~$G$.
  Since~$\Bb^+$ generates the cycle space of~$G^+$, it follows that~$f(\Bb^+)$ generates the cycle space of~$G$.
  Thus, to conclude, it suffices to show that~$\Bb$ generates~$f(\Bb^+)$.

  \begin{claim}\label{clm: Bgenerates}
    For any~$t \in T$ and \Fcycle $C \in \Bb_t$, the family~$\Bb$ generates~$f(C)$.
  \end{claim}
  \begin{claimproof}
    First, note that for all vertices $u,v\in V(G)$ and all nodes $t,t'\in V(T)$ such that $u,v$ belong to the adhesions of $t$ and $t'$, the \Fcycles of type $1$ in $\Bb$ generate the cycle $P_{t,u,v}\oplus P_{t',u,v}$, and also the \Fcycle $\sg{uv}\oplus P_{t,u,v}$ when $uv\in E(G)$.

    Consider now~$C \in \Bb_t$, and the corresponding \Fcycle $\tC \in \Bb$ of type~2.
    If~$uv$ is an edge of~$C$ contained in some adhesion, then in~$\tC$, $uv$ is replaced by some path~$P_{t',u,v}$.
    In~$f(C)$ on the other hand, $uv$ is either kept as is (if it is an edge of~$G$), or replaced by a possibly different path~$P_{t'',u,v}$.
    Thus, with regards to the handling of~$uv$, $\tC$ and~$f(C)$ differ by either $P_{t',u,v} \oplus P_{t'',u,v}$ or $\sg{uv} \oplus P_{t',u,v}$,
    which are generated by \Fcycles of type~1.
    Applying this reasoning to all such edges~$uv$ of $C$, we obtain that~$\tC$ can be transformed into~$f(C)$ by adding to it \Fcycles of type~1, proving the claim.
  \end{claimproof}
  Claim \ref{clm: Bgenerates} thus implies that~$\Bb$ generates~$f(\Bb^+)$, which by Claim \ref{clm: gen-torso} generates the cycle space of~$G$. This implies that
  $\Bb$ indeed generates the cycle space of $G$, and concludes the proof of \cref{thm:bn-tree-decomp-paths}.
\end{proof}

\section{Path-decompositions}
\label{sec: pw}
The main result of this section bounds the basis number of any graph with a path-decompositions which has both bounded adhesions, and parts of bounded basis number.
\begin{theorem}\label{thm:bn-path-decomp}
  Let $b,k\in\mathbb N$,~$G$ be a graph and $(P,\beta)$ be a path-decomposition of~$G$ with adhesion at most~$k$.
  Suppose that for any bag~$\beta(t)$ of~$(P,\beta)$, and any subset~$A \subseteq \beta(t)$, the subgraph~$G[A]$ has basis number at most~$b$, with $b \ge 1$.
  Then $\bn(G) \le b \cdot 2^{2^{O(k^2)}}$.
\end{theorem}

In a path-decomposition of width~$k$, adhesions have size at most~$k$, and bags and their induced subgraphs have order at most~$k+1$, and thus they all have basis number at most $b\eqdef 2k$ by Lemma \ref{lem:bn-edits}.
Thus \cref{thm:bn-path-decomp} immediately implies the following for bounded pathwidth graphs.
\begin{corollary}\label{coro:bn-pathwidth}
  For every $k\geq 0$, every graph with pathwidth at most~$k$ has basis number~$2^{2^{O(k^2)}}$.
\end{corollary}
\Cref{thm:bn-path-decomp,coro:bn-pathwidth} have been proved independently by Miraftab, Morin and Yuditsky~\cite{miraftab2026pathwidth}, who moreover improved the bounds to polynomial ones, see \cref{thm:pathwidth-linear} and \cref{thm:pd-linear}.
In later sections, we will prefer to refer to their stronger results to obtain much better bounds,
but nonetheless include this section for completeness, and also since our approaches differ significantly.

Our proof follows  closely the framework of \cite[Section~4]{mimi}. The key idea is that in a path-decomposition of adhesion~$k$,
one can encode with a semigroup of bounded size all the relevant information concerning how the at most~$2k$ vertices in the two adhesions of a given bag connect to each other.
Using Simon's Factorisation Forests Theorem~\cite{simon1990factorization}, one designs a bounded depth induction in which a given path-decomposition is decomposed into simpler (sub-)path-decompositions.
We will show in Section \ref{sec: proof-pw} that at each step of this induction, after ignoring a bounded number of problematic vertices,
one can construct a family of paths identifying the adhesions of the relevant path-decomposition
such that each path of the family intersects few bags, implying that the constructed family has bounded congestion.
This will allow us to apply \cref{thm:bn-tree-decomp-paths} in order to derive \cref{thm:bn-path-decomp}.

\subsection{Simon's Factorisation Forests Theorem}
\label{sec: Simon}
Let us recall the key notions defined in~\cite{mimi}.

\paragraph{Bi-interface graphs}
A \emph{bi-interface graph}~$\G$ of \emph{arity~$k$} consists of an underlying graph~$G$
and two partial injective maps $\lmap,\rmap : [k] \to V(G)$, called \emph{left and right interface maps}. 
For each $i \in [m]$, the vertices~$\lmap(i),\rmap(i)$ are called the \emph{$i$-th left and right interface vertices} respectively, and the images of $\lmap$ and $\rmap$ are called the \emph{left and right interfaces} respectively. 
We require that when $\lmap(i)$ and $\rmap(j)$ are defined for some $i,j\in [k]$ such that $\lmap(i) = \rmap(j)$, then~$i=j$.

Consider~$\G_1,\G_2$ two bi-interface graphs with same arity and interface \break maps~$\lmap_i,\rmap_i$ for~$i=1,2$.
The \emph{glueing} $\G_1 \odot \G_2$ is the bi-interface graph obtained as follows:
for the underlying graph of $\G_1\odot\G_2$, take the disjoint union of the underlying graphs of~$\G_1$ and~$\G_2$,
and identify the right interface of~$\G_1$ with the left interface of~$\G_2$,
i.e.\ identify vertices~$\rmap_1(i)$ with~$\lmap_2(i)$ whenever both are defined.
Finally, take~$\lmap_1$ and~$\rmap_2$ as left and right interface maps (see Figure \ref{fig: glueing}).

\begin{figure}[htb]
  \centering  
  \includegraphics[scale=0.65]{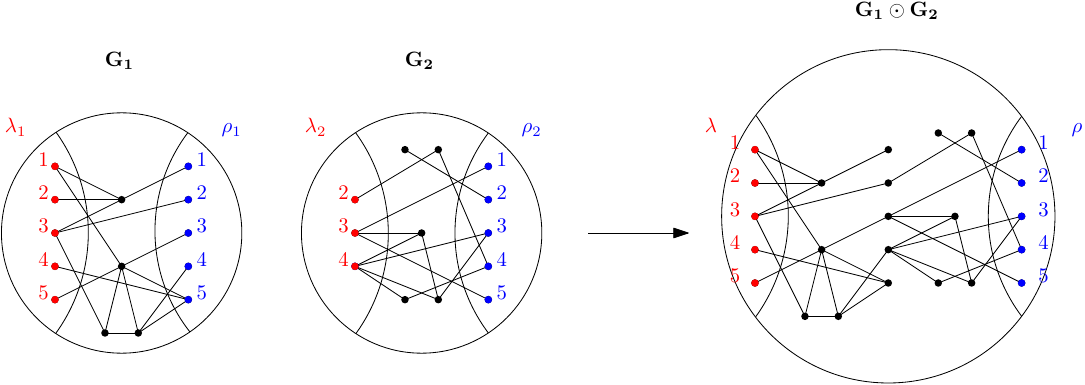}
  \caption{Glueing of two bi-interface graphs. On the left two bi-interval graphs $\G_1, \G_2$ of arity $5$. On the right their glueing. The numbers correspond to the preimages of the interface vertices by the interface maps.}  
  \label{fig: glueing}
\end{figure}

This glueing operation is associative, giving the set of bi-interface graphs of arity~$k$ the structure of a semigroup.
The next lemma shows that decomposing a bi-interface graph into a sequence of glueings is essentially the same as decomposing its underlying graph with a path-decomposition.
\begin{lemma}[{cf.\ \cite[Lemma~4.6]{mimi}}]
  \label{lem:path-decomp-glueing}
  For a graph~$G$ and a family~$X_1,\dots,X_m$ of subsets of vertices of $G$, the following are equivalent:
  \begin{itemize}
    \item $X_1,\dots,X_m$ (in that order) are the bags of a path-decomposition of~$G$ with adhesion at most~$k$.
    \item $G$ is the underlying graph of some glueing $\G_1 \odot \dots \odot \G_m$ of bi-interface graphs with arity~$k$,
      so that for each $i\in [m]$, $G[X_i]$ is the underlying graph of $\G_i$.
  \end{itemize}
\end{lemma}
\begin{proof}
  If~$G$ is the underlying graph of $\G_1 \odot \dots \odot \G_m$ and~for each $i\in [m]$, ~$X_i$ is the set of vertices coming from~$\G_i$,
  then it is easy to check that defining~$X_i$ to be the bag of the $i$-th node gives a path-decomposition of~$G$.
  The adhesion~$X_i \cap X_{i+1}$ is contained within the right interface of~$\G_i$ (and the left interface of~$\G_{i+1}$), hence it has size at most~$k$.

  Conversely, assume that $X_1,\dots,X_m$ corresponds to the sequence of the bags of a path-decomposition of $G$,
  and that the adhesion $A_i \eqdef X_{i} \cap X_{i+1}$ has size at most~$k$ for each $i\in [m-1]$. 
  Observe that one can define injective maps $f_i : A_i \to [k]$ for~each $i\in [m-1]$ so that whenever $x \in A_i \cap A_{i+1}$ for some $i\in [m-2]$, we have $f_i(x) = f_{i+1}(x)$. We also let $f_0, f_m$ denote the maps with signature $\emptyset\to [k]$.
  For each $i\in [m]$,
  define now the bi-interface graph~$\G_i$ with underlying graph $G[\bag(t_i)]$,
  where the left and right interface maps are given by the partial inverses of~$f_i$ and~$f_{i+1}$ respectively.
  Then one may check that~$G$ is the underlying graph of $\G_1 \odot \dots \odot \G_m$.
\end{proof}

In a bi-interface graph~$\G$ with underlying graph $G$ and  interface maps~$\lambda,\rho$,
a vertex~$v\in V(G)$ is called \emph{persistent} if it is both a left and a right interface vertex,
i.e.\ $\lambda(i) = \rho(i) = v$ for some~$i \in [k]$.
Observe that if~$v$ is persistent in~$\G$ and $\G = \G_1 \odot \G_2$, then~$v$ is also persistent in~$\G_1,\G_2$.

For every bi-interface graph $\G$ with underlying subgraph $G$ and every $X\subseteq V(G)$, we let $\G[X]$ denote denote the bi-interface graph with underlying subgraph $G[X]$, and whose interface maps are the restriction on the interface maps of $G$ on $X$. We also set $\G-X\eqdef \G[V(G)\setminus X]$. 
We denote by~$\hat{\G}$ the bi-interface graph $\G-\Pi$, where $\Pi$ is the set of the persistent vertices of~$\G$.

\paragraph{Connectivity abstraction}
Given a bi-interface graph~$\G$ with underlying graph $G$ and interface maps $\lmap, \rmap$, its \emph{abstraction}~$\abstr{\G}$ is the bi-interface graph defined as follows.
Denote by $I(\G) \eqdef \lmap([k]) \cup \rmap([k])$ the set of interface vertices of~$\G$.
Then the underlying graph of $\abstr{\G}$ is the graph with vertex set~$I(\G)$, where we add an edge~$xy$ whenever $G$ contains an $xy$-path whose internal vertices are not in~$I(\G)$.
The interface maps of $\abstr{\G}$ are given by $\lambda, \rho$.

Let $\G_1, \G_2$ be two bi-interface graphs, and for each $i\in \sg{1,2}$, let $G_i, \lmap_i, \rmap_i$ denote respectively the underlying graph and the interface maps of $\G_i$. $\G_1, \G_2$ are \emph{isomorphic} if $\lmap_1$ and $\lmap_2$ (resp.\ $\rmap_1$ and $\rmap_2$) are defined on the same values, and if there exists a graph isomorphism $f:G_1\to G_2$ such that for each $i\in [k]$ such that $\lmap_1(i)$ is defined, $f(\lmap_1(i))=\lmap_2(i)$ and for each $i\in [k]$ such that $\rmap_1(i)$ is defined, $f(\rmap_1(i))=\rmap_2(i)$. 

We denote by~$\Abstr_k$ the set of abstractions of bi-interface graphs of arity~$k$, considered up to isomorphism.
Observe that $|\Abstr_k| = 2^{O(k^2)}$.
For~$\alpha,\beta \in \Abstr_k$, we define their product
\[ \alpha \cdot \beta \eqdef \abstr{\alpha \odot \beta}. \]
It is routine to check that this product is associative, hence~$(\Abstr_k,\cdot)$ is a semigroup, and that~$\abstr{\cdot}$ is a semigroup homomorphism, i.e.\ for every bi-interface graphs $\G_1, \G_2$ we have\
\begin{equation}
  \abstr{\G_1 \odot \G_2} = \abstr{\G_1} \cdot \abstr{\G_2}.
\end{equation}

\paragraph{Factorisation forests}
A key idea of~\cite{mimi} is to use Simon's Factorisation Forests Theorem, a result from semigroup theory, in order to find some regular structure in a given path-decomposition.
The setting is as follows.
One considers a (possibly infinite) alphabet~$\Sigma$, a finite semigroup~$(S,\cdot)$, and a map $h : \Sigma \to S$.
We denote by~$\Sigma^+$ the set of finite non-empty sequences of elements of~$\Sigma$, called \emph{words}, and denote by~$\cdot$ the concatenation operation, which makes~$\Sigma^+$ a free semigroup.
The map~$h$ naturally extends to a semigroup homomorphism $(\Sigma^+,\cdot) \to (S,\cdot)$.

Given a word $w \in \Sigma^+$, the Factorisation Forests Theorem gives a way to split~$w$ down to single letters in a bounded number of steps, subject to some regularity conditions over~$h$.
Two ways of splitting~$w$ are considered:
\begin{description}
  \item[binary factorisation] $w = w_1 \cdot w_2$ for some $w_1,w_2\in \Sigma^+$
  \item[$h$-idempotent factorisation] $w = w_1 \cdot \ldots \cdot w_m$ for some $w_1,\ldots,w_m\in \Sigma^+$ such that all factors have the same image $e = h(w_i)$, which furthermore is idempotent, meaning $e \cdot e = e$.
\end{description}
The \emph{$h$-rank} of a word $w \in \Sigma^+$ is defined recursively as follows:
letters $a \in \Sigma$ have $h$-rank $0$,
and a word $w \in \Sigma^+ \setminus \Sigma$ has $h$-rank at most~$d$ if there is either a binary or an $h$-idempotent factorisation $w = w_1 \ldots w_m$ in which each factor~$w_i$ has $h$-rank at most~$d-1$.
This can be rephrased in terms of factorisation trees of depth~$d$, in which the leaves are letters, and each internal node corresponds to a binary or idempotent factorisation.

\begin{theorem}[Factorisation Forests Theorem~\cite{simon1990factorization,kufleitner2008height}]
  \label{thm:simon}
  Consider a finite semigroup~$(S,\cdot)$, an alphabet~$\Sigma$, and a semigroup homomorphism $h : (\Sigma^+,\cdot) \to (S,\cdot)$.
  Then any word $w \in \Sigma^+$ has $h$-rank at most~$3|S|$.
\end{theorem}

We will apply \cref{thm:simon} to glueings of bi-interface graphs, with the connectivity abstraction as homomorphism. More precisely, in what follows, we fix the alphabet~$\Sigma$ as the set of all bi-interface graphs of arity~$k$ (for arbitrary fixed~$k$).
Thus a word in~$\Sigma^+$ is a sequence $\G_1 \cdots \G_m$, which thanks to \cref{lem:path-decomp-glueing} can be seen as a path-decomposition of $\G_1 \odot \dots \odot \G_m$.
The role of the finite semigroup~$(S,\cdot)$ will be played by~$(\Abstr_k,\cdot)$, and we fix the semigroup homomorphism $h:\Sigma^+\to (\Abstr_k,\cdot)$ by setting $h(\G) \eqdef \abstr{\G}$ for every bi-interface graph $\G$.
Throughout the rest of this section, $h$-rank and idempotence conditions are understood relative to this specific homomorphism.

\subsection{Proof of Theorem~\ref{thm:bn-path-decomp}}\label{sec: proof-pw}
Following the techniques of~\cite{mimi}, we will prove \cref{thm:bn-path-decomp} by induction on the $h$-rank of the word of $\Sigma^+$ associated to the given path-decomposition.

We first extend the definition of the basis number on bi-interface graphs, by setting $\bn(\G)\eqdef \bn(G)$ for every bi-interface graph~$\G$ with underlying graph~$G$.
For brevity, throughout this section we adopt the convention that a bi-interface graph is denoted in bold-face as~$\G$,
and the corresponding normal font~$G$ refers to its underlying graph.

Recall that a vertex in~$\G$ is \emph{persistent} if it is present in both interfaces, and that~$\hat{\G}$ is obtained from~$\G$ after removing all persistent vertices.
We show that binary and idempotent factorisations both preserve bounded basis number, when ignoring persistent vertices.
The idempotent case is the core of this proof.%
\begin{lemma}[Binary case]
  \label{lem:bn-binary}
  Let $\G_1,\G_2$ be two bi-interface graphs of arity~$k$, with $\bn(\hat{\G_1}), \bn(\hat{\G_2}) \le b$, and $\G = \G_1 \odot \G_2$. Then we have
  \[ \bn(\hat{\G}) \le 2b + O(k). \]
\end{lemma}
\begin{proof}
  Let~$\Pi$ denote the set of persistent vertices in~$\G$, and for each $i\in \sg{1,2}$,  define $\G'_i \eqdef \G_i - \Pi$.
  Note that every vertex in~$\Pi$ is also persistent in both~$\G_1$ and~$\G_2$.
  Thus, for each $i\in \sg{1,2}$, the underlying graph~$G'_i$ contains~$\hat{G_i}$ as an induced subgraph.
  On the other hand, the vertices missing in~$\hat{\G_i}$ compared to~$\G'_i$ are persistent interface vertices of $\G'_i$, of which there are at most~$k$.
  Thus one can obtain~$G'_i$ from~$\hat{G_i}$ by adding back at most~$k$ vertices.
  It follows by \cref{lem:bn-edits} that for each $i\in \sg{1,2}$, 
  \[ \bn(G'_i) \le \bn(\hat{G_i}) + 2k\leq b+2k. \]
  Finally, we have $\hat{\G} = \G'_1 \odot \G'_2$,
  which means that the underlying graph~$\hat{G}$ has a separation of order at most~$k$ whose two sides are~$G'_1$ and~$G'_2$.
  By \cref{lem:bn-small-separator}, we then obtain
  \begin{align*}
    \bn(\hat{G}) & \le \bn(G'_1) + \bn(G'_2) + O\left(\log^2k\right) \\
                 & \le 2b + O(k). \qedhere
  \end{align*}
\end{proof}

\begin{lemma}[Idempotent case]
  \label{lem:bn-idempotent}
  Consider bi-interface graphs $\G_1,\dots,\G_m$ of arity~$k$, and $\G = \G_1 \odot \dots \odot \G_m$.
  Assume that all $\G_i$ have the same connectivity abstraction $\abstr{\G_1} = \dots = \abstr{\G_m}$, which furthermore is idempotent in $(\Abstr_k, \cdot)$. 
  Moreover, assume that $\bn(\hat{\G_i}) \le b$ for each~$i\in [m]$.
  Then
  \[ \bn(\hat{\G}) \le O\left(b \cdot k^2 \log^2k\right). \]
\end{lemma}
\begin{proof}
  Assume that in~$\G_i$, the~$j$-th interface vertex, say~$v$, is persistent (i.e.\ the~$j$-th left and~$j$-th right interface vertices are equal).
  Observe that this fact is encoded in the connectivity abstraction~$\abstr{\G_i}$.
  Then by idempotence, the~$j$-th interface vertices of all~$\G_{i'}$ are also persistent.
  In turn, this implies that the~$j$-th interface vertex of~$\G$ is persistent, and corresponds to~$v$.
  It follows from these remarks that
  \begin{equation}
    \label{eq:del-persistent}
    \hat{\G} = \hat{\G_1} \odot \dots \odot \hat{\G_m}.
  \end{equation}
  In a sense, \eqref{eq:del-persistent} tells us that we can reason entirely in the graphs~$\hat{\G_i}$ without persistent vertices.
  Let~$\Pi$ denote the set of persistent vertices of~$\G$.

  For each $i\in [m]$, we let~$X_i$ be the vertex set of the underlying graph~$G_i$, and
  for each $i\in \sg{2,\ldots,m}$, we set $I_i \eqdef X_{i-1} \cap X_i$.
  In the corresponding path-decompositions, $X_i$ is a bag and~$I_i$ an adhesion.
  Further, let~$\hat{X_i},\hat{I_i}$ be the corresponding sets with persistent vertices~$\Pi$ removed.
  Note that the subsets~$\hat{I_i}$ are pairwise disjoint, and~$\hat{X_i},\hat{X_j}$ are disjoint unless $|i-j| \le 1$.

  The following is the key fact, and the entire reason for using connectivity abstractions and Simon's idempotence condition.
  \begin{claim}[{Cf.~\cite[Lemma~4.10]{mimi}}]
    \label{clm:idempotent-short-paths}
    Let $i\in \sg{2,\ldots,m}$. If~$x,y \in I_i$ are connected by a path~$Q$ in~$G$, then they are also connected by a path~$Q'$ which is contained in $X_{i-1} \cup X_i$,
    and which furthermore satisfies $V(Q) \cap I_i = V(Q') \cap I_i$.
  \end{claim}
  \begin{claimproof}
    Define $\mathbf{L} \eqdef \G_1 \odot \dots \odot \G_{i-1}$ and $\mathbf{R} \eqdef \G_i \odot \dots \odot \G_m$.
    The idempotence assumption implies that
    \[ \abstr{\mathbf{L}} = \abstr{\mathbf{R}} = \abstr{\G_{i-1}} = \abstr{\G_i}. \]
    It follows that, if two vertices~$x,y \in I_i$ can be connected by a path in~$\mathbf{L}$ (resp.~$\mathbf{R}$),
    then they can also be connected by a path in~$\G_{i-1}$ (resp.~$\G_i$).
    The same remains true if these paths are forbidden to have internal vertices in~$I_i$.

    Consider now an $xy$-path~$Q$.
    It can be split into a sequence of subpaths $P_1,\dots,P_t$ where each~$P_j$ has its endpoints in~$I_i$ and its internal vertices in either~$\mathbf{L} \setminus I_i$ or~$\mathbf{R} \setminus I_i$.
    Then, each~$P_j$ can be replaced by an equivalent path in~$G_{i-1}$ or~$G_i$, giving a path with the desired properties.
  \end{claimproof}

  We will use \cref{clm:idempotent-short-paths} to construct a path family  with low congestion capturing the adhesions of the given path-decomposition, so as to apply \cref{thm:bn-tree-decomp-paths}.
  Precisely, consider a connected component~$C$ of the underlying graph~$\hat{G}$.
  Using \cref{lem:path-decomp-glueing}, we obtain a path-decomposition $(P,\beta)$ of~$\hat{G}[C]$
  with bags
  \begin{equation}
    \label{eq:path-decomp-component}
    \hat{X_1} \cap C, \dots, \hat{X_m} \cap C,
  \end{equation}
  and whose adhesions are the sets~$\hat{I_i} \cap C$ for $1<i\le m$. In particular, all adhesions have size at most~$k$.

  Let $1<i\le m$ and consider any pair $x,y \in \hat{I_i} \cap C$ of distinct vertices.
  By definition, they are connected by a path~$Q$ in~$\hat{G}$,
  thus \cref{clm:idempotent-short-paths} gives that they are also connected by a path~$Q'$ in $X_{i-1} \cup X_i$ with $V(Q') \cap I_i = V(Q) \cap I_i$.
  Note that the persistent vertices of~$\G$ are all contained in~$I_i$.
  Thus, since~$Q$ avoids all persistent vertices, so does~$Q'$, meaning that~$Q'$ is in fact contained in~$\hat{G}$, and thus in the connected component~$C$.

  Therefore, for any $1<i\le m$ and any~$x,y$ in the adhesion~$I_i \cap C$,
  there is a path~$P_{i,x,y}$ connecting~$x$ to $y$ in~$\hat{G}$, which furthermore is contained in $\hat{X_{i-1}} \cup \hat{X_i}$. We let $\Pc$ denote the family of such paths~$P_{i,x,y}$.
  Clearly, $\Pc$ captures the adhesions of $(P,\beta)$.
  Recall that $\hat{X_i},\hat{X_j}$ are disjoint whenever $|i-j| \ge 2$.
  Thus, for every $1\leq j\le m$, the only paths of~$\Pc$ that may intersect $\hat{X_j}$ are of of the form~$P_{i,x,y}$ for some $j-1 \le i \le j+2$. As for each~$1<i\leq m$, $\Pc$ contains at most~$\binom{k}{2}$ paths of the form~$P_{i,x,y}$, the edge congestion of~$\Pc$ is then at most~$4\binom{k}{2} \le 2k^2$.
  
  In order to apply Theorem \ref{thm:bn-tree-decomp-paths}, it remains to bound the basis-number of the torsos of $(P,\beta)$. For this, we fix $i\in [m]$ and let $H_i$ denote the torso associated to the bag $\hat{X_i}\cap C$ of $(P,\beta)$. 
  First, note that $H_i$ is obtained from the graph $\hat{G_i}[C]$ by adding at most $2\binom{k}{2} \le k^2$ edges between interface vertices.
  Moreover, since~$C$ is a connected component of~$\hat{G}$,
  the restriction~$\hat{G_i}[C]$ is the union of some connected components of~$\hat{G_i}$, which immediately gives
  \[ \bn(\hat{G_i}[C]) \le \bn(\hat{G_i}) \le b. \]
  Thus, by \cref{lem:bn-edits}, after adding the~$k^2$ edges between interface vertices we have $\bn(H_i)\leq b + O\left(\log^2k\right)$.
  
  We have thus shown that~$(P, \beta)$ is a path-decomposition of~$\hat{G}[C]$ whose torsos have basis number at most $b+O\left(\log^2k\right)$,
  and whose adhesions are captured by a path system of congestion~$2k^2$.
  By \cref{thm:bn-tree-decomp-paths}, this implies that
  \begin{align*}
    \bn(\hat{G}[C]) & = (4k^2 + 1)\left(b + O(\log^2k)\right) \\
                    & = O\left(b\cdot k^2 \log^2k\right).
  \end{align*}
  Since this holds for each connected component~$C$ of~$\hat{G}$, we obtain the same bound for~$\bn(\hat{G})$ itself.
\end{proof}

With these two lemmas, we obtain a bound by induction on the $h$-rank.
\begin{lemma}
  \label{lem:bn-hrank}
  Consider $\G = \G_1 \odot \dots \odot \G_m$, where $\G_1,\dots,\G_m$ are bi-interface graphs with arity~$k$,
  such that the word $\G_1 \cdots \G_m$ has $h$-rank at most~$d$.
  Assume that for each~$i$, $\bn(\hat{\G_i}) \le b$, and $b \ge 1$. Then $\bn(\hat{\G}) \le b \cdot 2^{O(d \log k)}$.
\end{lemma}
\begin{proof}
  Let~$f(k)$ denote the polynomial bound implicitly given by \cref{lem:bn-binary,lem:bn-idempotent}.
  We prove by induction on~$d$ that $\bn(\hat{\G}) \le b \cdot (f(k))^d$.
  Since~$f$ is polynomial, this gives $\bn(\hat{\G}) \le b \cdot 2^{O(d \log k)}$.
  In the base case $d=0$, the word $\G_1 \cdots \G_m$ consists of a single letter, i.e.~$m=1$ and $\G = \G_1$.
  We thus have $\bn(\hat{\G})\leq b$ as desired.

  Assume now that~$d>0$.
  By definition, the word $w \eqdef \G_1 \cdots \G_m$ has a factorisation $w = w_1 \cdots w_\ell$ that is either binary or $h$-idempotent,
  in which each~$w_j$ has $h$-rank at most~$d-1$.
  If $w_j = \G_{i_j} \cdots \G_{i_{j+1}-1}$, consider the glueing $\G'_j \eqdef \G_{i_j} \odot \dots \odot \G_{i_{j+1}-1}$.
  Note that $\G = \G'_1 \odot \dots \odot \G'_\ell$.
  Since~$\G'_j$ has $h$-rank at most~$d-1$,
  the induction hypothesis gives $\bn(\hat{\G'_j}) \le b \cdot (f(k))^{d-1}$ for each $j\in [\ell]$.
  By \cref{lem:bn-binary,lem:bn-idempotent} (for the appropriate choice of~$f$), if $\bn(\hat{\G'_j}) \le b'$ for all~$j$, then $\bn(\hat{\G}) \le b' \cdot f(k)$.
  Thus $\bn(\hat{G}) \le (f(k))^d$ as desired.
\end{proof}

We finally obtain the main result of this section.
\begin{proof}[Proof of \cref{thm:bn-path-decomp}]
  Consider~$G$ with a path-decomposition~$(P,\beta)$, with adhesion at most~$k$, and call $t_1,\ldots,t_m$ the nodes of~$P$, appearing in this order.
  Assume that for any~$i$ and any subset $A \subseteq \beta(t_i)$, we have $\bn(G[A]) \le b$.

  By \cref{lem:path-decomp-glueing}, $G$ is the underlying graph of a glueing $\G=\G_1 \odot \dots \odot \G_m$ of bi-interface graphs with arity~$k$,
  such that for each $i\in [m]$, the underlying graph of~$\G_i$ is~$G[\beta(t_i)]$.
  In particular, for each $i\in [m]$, we have~$\bn(\hat{\G_i}) \le b$.
  \Cref{lem:bn-hrank} then gives $\bn(\hat{\G}) \le b \cdot 2^{O(d \log k)}$, where~$d$ denotes the $h$-rank of~$\G_1 \cdots \G_m$.
  As~$G$ has at most $k$ vertices which are persistent in $\G$, adding them back using \cref{lem:bn-edits} increases the basis number by at most~$2k$.
  Finally, by \cref{thm:simon}, we have $d\leq 3|\Abstr_k|=2^{O(k^2)}$, and thus
  \[ \bn(G) \le b \cdot 2^{O(d \log k)} + 2k \le b\cdot 2^{2^{O(k^2)}}. \qedhere \]
\end{proof}

\section{The bounded treewidth case}
\label{sec:tw}
This section proves that bounded treewidth graphs have bounded basis number.%
\twbn*
The proof of \cref{thm:bn-treewidth} relies on a technical result of Bojańczyk and Pilipczuk, namely \cite[Lemma~5.9]{mimi},
which implies that any bounded treewidth graph has a tree-decomposition satisfying condition~\ref{item:path-system} of \cref{thm:bn-tree-decomp-paths}, in which torsos have bounded pathwidth (see Lemma \ref{lem:magic-decomp}).
\Cref{sec: td2} will prove technical generalisations of these results, needed to establish the more general \cref{thm: main-td}.

Theorem \ref{thm:bn-treewidth} easily follows from the next result, combined with everything we proved so far (and Theorem \ref{thm:pathwidth-linear} for a polynomial bound). 
\begin{lemma}\label{lem:magic-decomp}
  Let~$G$ be a connected graph with treewidth~$k$.
  Then~$G$ has a tree-decomposition~$(T,\beta)$ whose torsos have pathwidth at most~$3k+1$,
  and whose adhesions are captured by a family~$\Pc$ of paths with congestion~$O(k^4)$.
\end{lemma}

\begin{proof}[Proof of \Cref{thm:bn-treewidth} using \Cref{lem:magic-decomp}]
 Let $(T,\beta)$ be the tree-decomposition given by \cref{lem:magic-decomp}.
 The torsos of~$(T,\beta)$ have pathwidth at most~$3k+1$, hence basis number at most~$12k+4$ by \cref{thm:pathwidth-linear}.
 With the path system of congestion~$O(k^4)$, \cref{thm:bn-tree-decomp-paths} then gives $\bn(G) = O(k^5)$.
\end{proof}

The remainder of this section gives the precise statement of \cite[Lemma~5.9]{mimi},
and shows through routine checks that it implies \cref{lem:magic-decomp}.
As the statement is quite technical, we will faithfully reproduce it below.
One definition is needed first.

Recall that the tree-decompositions we consider are rooted.
Let $(T,\beta)$ be a tree-decomposition of a graph $G$, and let $X\subseteq V(T)$ be a subset of nodes that contains the root $r$ of $T$.
The quotient tree~$T/X$ is the tree with nodes~$X$ such that the ancestor--descendant relation in~$T/X$ is the same as in~$T$, restricted to~$X$.
The \emph{quotient} of the decomposition $(T,\beta)$ with respect to $X$ is the tree-decomposition $(T/X, \beta_X)$ of $G$, where for each $t\in X$, $\beta_X(t)$ is the union of all bags $\beta(s)$ from $(T, \beta)$ such that $s$ is a descendant of $t$ (including $t$ itself), and such that $t$ is the closest ancestor of $s$ in $X$ (see Figure \ref{fig: quotient}).
Observe that for any~$x \in X$, the $T$-adhesion of~$x$ and its $T/X$-adhesion are equal,
and similarly the $T$-component and $T/X$-component are equal.
 
\begin{figure}[htb]
  \centering  
  \includegraphics[scale=0.63]{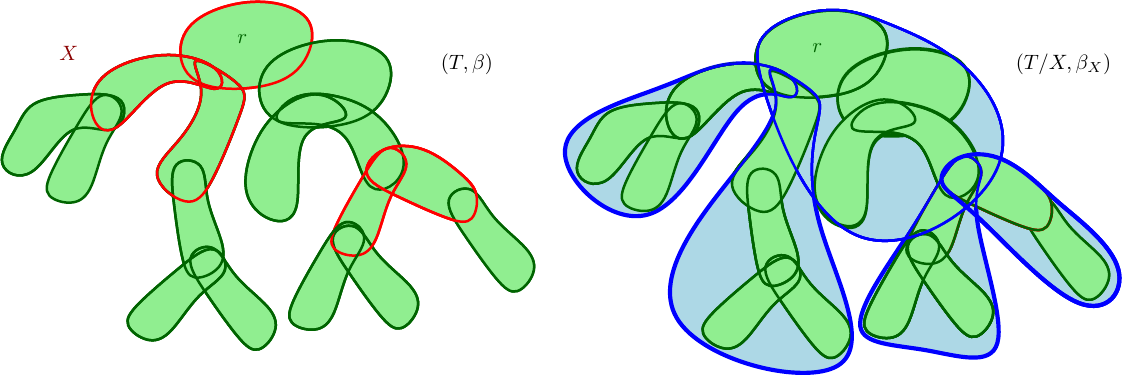}
  \caption{Left: In green, a tree-decomposition $(T,\beta)$, and in red, the bags corresponding to a subset $X\subseteq V(T)$. Right: In blue, the quotient $(T/X, \beta_X)$.}  
  \label{fig: quotient}
\end{figure}

\begin{lemma}[{\cite[Lemma~5.9]{mimi}}]
  \label{lem: fameux59}
  Let~$(T, \beta)$ be a width~$k$ sane tree-decomposition of a connected graph~$G$.
  Then one can find a set of nodes~$X$ in~$T$, which includes the root of~$T$,
  and families of paths $\{\Pc_x\}_{x \in X}$, such that every $x \in X$ satisfies the following.
  \begin{enumerate}
    \item \label{item:margin-pw} The $T/X$-marginal graph of~$x$ has pathwidth at most~$2k+1$.
    \item Every element $P \in \Pc_x$ is a path in~$G$ that satisfies:
      \begin{enumerate}
        \item \label{item:path-component} except for its endpoints, $P$ visits only vertices in the $T$-component of~$x$, and
        \item if $y \in X$ is a strict descendant of~$x$, then the restriction of~$P$ to the $T$-component of~$y$ yields a subpath of~$P$.
      \end{enumerate}
    \item \label{item:path-endpoints} All paths in~$\Pc_x$ have the same source, which belongs to the $T$-margin of~$x$.
      Conversely, each vertex of the $T$-adhesion of~$x$ is the target of some path from~$\Pc_x$.
    \item \label{item:paths-congestion} The following set has size at most~$2k^3$:
      \begin{align*}
      \mathrm{load}_x \eqdef \{(P,y): ~& P \in \Pc_y,\\
      & \text{$y \in X$ is a strict ancestor of~$x$} \\ & \text{and~$P$ intersects the $T$-component of~$x$}\}.
      \end{align*}
  \end{enumerate}
\end{lemma}

\begin{proof}[Proof of \cref{lem:magic-decomp}]
  Consider~$G$ a connected graph of treewidth~$k$.
  Then by \cref{lem:sane}, $G$ has a sane tree-decomposition $(T,\beta)$ of width~$k$.
  Applying \cref{lem: fameux59} to~$(T,\beta)$ yields a set of nodes~$X$ and path families~$\sg{\Pc_x}_{x \in X}$.
  Without loss of generality, each~$\Pc_x$ consists of at most~$k$ paths~$P_{x,u}$ where~$u$ ranges over the adhesion of~$x$, and the target of~$P_{x,u}$ is~$u$.
  Our goal is to show that~$(T/X,\beta_X)$ satisfies the conclusion of \cref{lem:magic-decomp}.

  First, recall that for every $x\in X$, the $T/X$-marginal graph of $x$ is obtained from its $T/X$-torso by removing the at most~$k$ vertices from the adhesion of~$x$.
  \Cref{lem: fameux59} gives that the marginal graph of~$x$ has pathwidth at most~$2k+1$,
  and thus after adding back the adhesion vertices, the $T/X$-torso of~$x$ has pathwidth at most~$3k+1$ as desired.

  We now construct a family $\Pc$ of paths with bounded congestion that captures the adhesions of $(T/X,\beta_X)$.
  For every node $x\in X$, and every pair $u,v$ of distinct vertices in the $T/X$-adhesion of $x$,
  consider $P_{x,u},P_{x,v} \in \Pc_x$ the two paths having respectively $u$ and $v$ as targets.
  By condition~\labelcref{item:path-endpoints} of \cref{lem: fameux59}, $P_{x,u}$ and $P_{x,v}$ have the same source.
  Thus we can choose some $uv$-path~$Q_{x,u,v}$ contained in~$P_{x,u} \cup P_{x,v}$.
  We let $\Pc$ denote the family of all the paths $Q_{x,u,v}$ for each $x\in X$ and every pair $u,v$ of distinct vertices in the $T/X$-adhesion of~$x$ (this family of paths is seen as a multiset, as in \cref{sec: td1}).
  Clearly, $\Pc$ captures the adhesions of $(T/X, \beta_X)$. We will now bound the congestion of $\Pc$.

  \begin{claim}
  \label{clm:load}
    For any vertex~$v \in V(G)$, there are at most~$2k^3+k$ paths in $\bigcup_{y \in X} \Pc_y$ containing~$v$ as a non-target vertex.
  \end{claim}
  \begin{claimproof}
    Consider the subtree $T_v$ of~$T/X$ consisting of all the nodes $z\in X$ whose bag~$\bag(z)$ contains~$v$, and call~$x$ its root (i.e.\ its vertex that is closest to the root of $T/X$).
    Note that~$v$ belongs to the $T/X$-margin of~$x$, and that $v$ belongs to the adhesions of all strict descendants of $x$ in $T_v$.
    We claim that if for some $y\in X$, some path~$P \in \Pc_y$ contains~$v$ and~$v$ is not the target of~$P$, then~$y$ is an ancestor of~$x$ (with possibly $y=x$).
    Indeed, if $v$ is not the target of $P$, then Conditions~\labelcref{item:path-component} and~\labelcref{item:path-endpoints} of \cref{lem: fameux59} give that $v$ belongs to the component of~$y$. This implies that every bag of $(T/X, \beta_X)$ that contains~$v$ is a descendant of~$y$, which is in particular the case of~$x$.

    Thus, if~$P \in \Pc_y$ contains~$v$ as a non-target vertex, either~$x=y$ or~$y$ is a strict ancestor of~$x$, and in the latter case we have~$(P,y) \in \textrm{load}_x$.
    As there are at most~$2k^3$ pairs in~$\textrm{load}_x$, and at most~$k$ paths in~$\Pc_x$, this proves the claim.
  \end{claimproof}
  
  It follows from Claim \ref{clm:load} that each edge~$uv \in E(G)$ is contained in at most~$4k^3+2k$ paths in $\bigcup_{y \in X} \Pc_Y$.
  Indeed, each path~$P$ using the edge~$uv$ contains either~$u$ or~$v$ as a non-target vertex.

  Finally, observe that for each $x\in X$, each path~$P_{x,u} \in \Pc_x$ is used by at most~$k$ paths in~$\Pc$,
  namely those of the form~$Q_{x,u,v}$ with~$v$ in the adhesion of~$x$.
  It follows that each edge is used by at most~$k \cdot (4k^3 + 2k)$ paths in~$\Pc$, concluding the proof that $\Pc$ has congestion $O(k^4)$.
\end{proof}

\section{Tree-decompositions with large bags}
\label{sec: td2}
We now generalise the ideas of \cref{sec:tw} to show that under appropriate hypotheses, tree-decompositions with bags of unbounded size (but bounded adhesions) preserve bounded basis number.
Recall that a class of graphs $\Gc$ is \emph{monotone} if it is closed under taking subgraphs.

\tdbn*

In what follows, for a class of graphs $\Gc$ and some $k\geq 0$, we let $\Gc^{+k}$ denote the class of those graphs $G$ for which there exists a subset $A$ of vertices (called \emph{apices}) of size at most $k$, such that $G-A\in \Gc$.

To prove \cref{thm: main-td}, we want to generalise \cref{lem:magic-decomp} by allowing bags of unbounded size, and thus also \cite[Lemma~5.9]{mimi}.
Though technical, the proof is a fairly direct adaptation of the arguments of \cite[Section~5]{mimi},
by replacing conditions of the form `the tree-decomposition~$(T,\beta)$ has bounded width'
by `torsos of~$(T,\beta)$ are in~$\Gc^{+O(k)}$'.
Precisely, we generalise \cref{lem:magic-decomp} as follows.
\begin{lemma}\label{lem:magic-decomp-general}
  Let $\Gc$ be a monotone class of graphs and $k\in \mathbb N$.
  Consider~$G$ a connected graph that admits a tree-decomposition with adhesion~$k$ and whose torsos are all in $\Gc$.
  Then~$G$ has a tree-decomposition~$(T,\beta)$ such that each torso has a path-decomposition with adhesion~$3k$ and parts in $\Gc^{+2k}$,
  and such that the adhesions of $(T,\beta)$ are captured by a family~$\Pc$ of paths with congestion~$O(k^4)$.
\end{lemma}
\begin{proof}[Proof of \Cref{thm: main-td} using \Cref{lem:magic-decomp-general}]
  In the tree-decomposition~$(T,\beta)$ given by \cref{lem:magic-decomp-general}, consider a torso~$H$.
  It has a path-decomposition with parts in~$\Gc^{+2k}$ and adhesion at most~$3k$.
  By \cref{lem:bn-edits}, graphs in~$\Gc^{+2k}$ have basis number at most~$b+4k$,
  hence \cref{thm:pd-linear} gives
  \[ \bn(H) \le b + 4k + O(k\log^2k) \le b + O(k\log^2k). \]
  Thus~$(T,\bag)$ is a path-decomposition with adhesions captured by a path system of congestion~$O(k^4)$, and torsos of basis number~$b+O(k\log^2k)$.
  \Cref{thm:bn-tree-decomp-paths} then gives $\bn(G) = O((b+k\log^2k)\cdot k^4)$.
\end{proof}

\Cref{lem:magic-decomp-general} follows from the next technical lemma, which is a generalisation of \cite[Lemma~5.9]{mimi}.
The statement requires a couple of simple definitions about set systems.

For a fixed ground set~$V$, a family $\Vc\subseteq 2^V$ of subsets of $V$ is \emph{downwards-closed} if it is closed under taking subsets.
Moreover, for every integer $k\geq 0$, we let $\mathcal V^{+k}$ denote the set of subsets $U$ of $V$ for which there exists some subset $A\subseteq U$ of size at most $k$ such that $U\setminus A\in \Vc$.

\begin{lemma}[{Generalisation of \cite[Lemma~5.9]{mimi}}]
 \label{lem: fameux59-general}
  Let~$(T, \beta)$ be a sane tree-decomposition of a connected graph~$G$ with adhesion~$k$, and let 
  $$\Vc\eqdef \sg{U\subseteq V(G): \exists t\in V(T), U\subseteq \beta(t)}$$
  denote the set of all subsets of bags of $(T,\beta)$.
  Then one can find a set of nodes~$X$ in~$T$ which includes the root of~$T$,
  and families of paths $\{\Pc_x\}_{x \in X}$, such that for every $x \in X$, the family $\Pc_x$ satisfies properties \ref{item:path-component} to \ref{item:paths-congestion} from Lemma \ref{lem: fameux59}, together with the following property.
  \begin{enumerate}
    \item[1']\label{item:margin-path-decomposition} The $T/X$-marginal graph of~$x$ admits a path-decomposition of adhesion at most $2k$, in which every bag belongs to $\Vc^{+k}$.
  \end{enumerate}
\end{lemma}
\begin{proof}[Proof of \Cref{lem:magic-decomp-general} using \cref{lem: fameux59-general}]
  The proof is largely the same as that of \cref{lem:magic-decomp}, but there are a few subtleties that require additional care.

  Firstly, we may assume that the given tree-decomposition is sane.
  Indeed, given any tree-decomposition~$(T,\bag)$, \cref{lem:sane,rem: sane-adhesion} give a sane tree-decomposition~$(T',\bag')$
  such that bags and adhesions of~$(T',\bag')$ are subsets of bags and adhesions respectively of~$(T,\bag)$.
  This implies that torsos of~$(T',\bag')$ are subgraphs of torsos of~$(T,\bag)$, hence are in~$\Gc$.

  Assuming now~$(T,\bag)$ to be sane, we apply \cref{lem: fameux59-general}, yielding nodes $X \subseteq T$ and path systems~$\{\Pc_x\}_{x \in X}$.
  Consider a $T/X$-marginal graph~$H$.
  We claim that for any $A \in \Vc$, we have $H[A] \in \Gc$.
  First, by definition~$A$ is contained in a bag~$\bag(t)$ for some~$t \in T$.
  Secondly, if~$uv$ is an edge of~$H[A]$ that is not in~$G$, then the pair~$\{u,v\}$ is contained in some adhesion of~$(T/X,\bag_X)$, which is also an adhesion of~$(T,\bag)$.
  \Cref{rmk:torso-adh} then shows that~$uv$ is also an edge in the $(T,\bag)$-torso of~$t$.
  This proves that~$H[A]$ is a subgraph of the torso of~$t$, and is in~$\Gc$.
  Now by assumption~$H$ has a path-decomposition with bags in~$\Vc^{+k}$, which means parts in~$\Gc^{+k}$.
  Recall that~$H$ is a $T/X$-marginal graph, i.e.\ a $T/X$-torso with the adhesion vertices removed.
  Adding back the~$k$ adhesion vertices, we obtain that each $T/X$-torso has a path-decomposition with adhesion at most~$3k$ and parts in~$\Gc^{+2k}$, as desired.

  It only remains to construct the path system~$\Pc$; we omit this part of the proof as it is exactly the same as in the case of \cref{lem:magic-decomp}.
\end{proof}

The rest of this section is dedicated to the proof of \cref{lem: fameux59-general}.
In \Cref{sec:hypergraph-defs}, we introduce the general definitions and notations from~\cite{mimi} used in the proof, mainly related to hypergraphs.
\Cref{sec: fameux59} explains how to generalise \cite[Lemma~5.9]{mimi}, keeping the same conventions and proof structure as in~\cite{mimi}.
We present the modified definitions and statements, but omit or only sketch many proofs, whenever they are the same as the original.
A single important technical lemma requires additional care, namely \cref{lem: 58-general} (corresponding to \cite[Lemma~5.8]{mimi}), in which careful modifications of path-decompositions occur.
\Cref{sec:thinness-substitution} is dedicated to the proof of the latter.

\subsection{Definitions}
\label{sec:hypergraph-defs}
\paragraph{Hypergraphs} A \emph{hypergraph} is a pair $H=(V(H), E(H))$, where $V(H)$ denote the \emph{vertices} of $H$, and $E(H)$ is a multiset of subsets of $V(H)$, called the \emph{hyperedges} of $H$. Note that in $E(H)$, we allow repetitions. We only consider finite hypergraphs here.
An hypergraph $H'$ is a subhypergraph of $H$ if $V(H') \subseteq V(H)$ and $E(H')\subseteq  E(H)$.
For a subset $X\subseteq V(H)$ of vertices, $H[X]$ denotes the subhypergraph \emph{induced} by $X$, whose vertex set is $X$, and which contains all the hyperedges $e$ of $H$ such that $e\subseteq X$.

A \emph{path} in $H$ is a sequence $P=(u_1, e_1, u_2, \ldots, u_p, e_p, u_{p+1})$ such that the $u_i$'s are pairwise distinct vertices, the $e_i$'s are pairwise distinct hyperedges, and for every $i\in \sg{1,\ldots, p}$, we have $u_i, u_{i+1}\in e_i$. When such a path exists, we say that $u_1$ and $u_{p+1}$ are \emph{connected} in $H$, and call $P$ a \emph{$u_1u_{p+1}$-path}. The notions of connectedness and connected components naturally generalise in hypergraphs.
Let $A,B$ be two subsets of vertices of a hypergraph $H$. We say that a subset~$X$ of vertices of $H$ \emph{separates} $A$ from $B$ if every path in $H$ connecting a vertex from $A\setminus X$ to a vertex from $B\setminus X$ contains a vertex in $X$. 

Tree-decompositions are defined for hypergraphs in the same way as for graphs, except that we ask instead of Item \ref{item: td2} that every hyperedge must be included in some bag. For a tree-decomposition $(T,\beta)$ of a hypergraph $H$, the \emph{part} of a node $t\in V(T)$ is the induced sub-hypergraph~$H[\beta(t)]$.

\paragraph{Hypertorsos, prefixes}
Let $(T,\beta)$ be a tree-decomposition of a graph (or hypergraph)~$G$.
For a node $t \in T$, the \emph{hypertorso} of~$t$ is defined as the hypergraph obtained from the part~$G[\beta(t)]$ by adding one hyperedge~$e_z$ equal to the adhesion of~$z$ for each child~$z$ of~$t$ (recall that~$T$ is rooted).
Note that several children of~$t$ can have the same adhesion set, hence the hypertorso can have multi-hyperedges.
Also, the hypertorso of~$t$ does not have an hyperedge for the adhesion of~$t$, only for the adhesions of its children.
The point is that, unlike the standard notion of torso, there is a one-to-one correspondence between hyperedges of the hypertorso that are not in~$G$, and adhesions of children of~$t$.
This is the entire motivation for using hypergraphs in this proof.

A \emph{prefix} of $T$ is a subset $Z$ of nodes of $T$ which is closed under taking ancestors. For a prefix $Z$ of $T$, let $\partial Z$ denote the set of nodes of $T-Z$ having their parent in~$Z$.
Let $\bag(Z) = \bigcup_{z \in Z} \bag(z)$ denote the union of bags of~$Z$.
We more generally define the hypertorso of the prefix~$Z$ as the hypergraph $\hyp(T,Z)$ obtained from the induced subgraph~$G[\bag(Z)]$
by adding one hyperedge $e_z$ equal to the adhesion of~$z$ in $(T,\beta)$ for each $z\in \partial Z$.

We denote by $\mathfrak{e}(T,\beta)$ the multiset of hyperedges $e\subseteq V(G)$ of all hypertorsos of $(T,\beta)$ that are not edges of~$G$.
In other words, $\mathfrak{e}(T,\beta)$ is the multiset of adhesions~$e_z$ of~$(T,\beta)$, seen as hyperedges.

\subsection{Adapting Lemma \ref{lem: fameux59}}
\label{sec: fameux59}
We sketch the proof of \cref{lem: fameux59-general} in the remainder in this section, following closely \cite[Section 5]{mimi}.
The first step is the following lemma, which allows to find two paths between two specified vertices
while in a sense spreading out the path congestion between different subtrees.

\begin{lemma}[{Generalisation of \cite[Lemma 5.2]{mimi}}]
\label{lem: fameux52-general}
  Let $(T, \beta)$ be a sane tree-decomposition of a connected graph $G$ with adhesion at most $k$,
  and $$\Vc\eqdef \sg{U\subseteq V(G): \exists t\in V(T), U\subseteq \beta(t)}$$ be the set of all subsets of the bags of $(T,\beta)$.
  Let $u,v$ be vertices in the root bag of $(T,\beta)$. Then, there exists a non-empty prefix $Z$ of $T$ with the following properties.
 \begin{enumerate}
  \item\label{item: hypertorso} $\hyp(T,Z)$ has a path-decomposition with adhesion at most~$2k$ whose bags are all in $\Vc^{+k}$, and
  \item\label{item: 2paths} there are two $uv$-paths~$P_1,P_2$ in $\hyp(T,Z)$ such that no edge of $\mathfrak{e}(T,\beta)$ belongs to both~$P_1$ and~$P_2$.
 \end{enumerate}
\end{lemma}

We note that the two paths from Condition~\ref{item: 2paths} might be equal, in which case they must be paths from~$G[Z]$.
In particular, if $z$ is the root of $T$, observe that if $u$ and $v$ are already connected by a path included in $G[\beta(z)]$, then the set $Z\eqdef \sg{z}$ satisfies trivially the properties of Lemma \ref{lem: fameux52-general}.

\paragraph{Networks, cutedges}
As in \cite{mimi}, a \emph{network} $\Hbb$ is defined as a connected hypergraph $H$ with two distinguished vertices: a \emph{source} $s$ and a \emph{sink} $t$. A \emph{cutedge} in $\Hbb$ is an hyperedge $e\in E(H)$ appearing in every $st$-path. Note that every cutedge must have multiplicity $1$ in $E(H)$.

The following is a simple consequence of Menger's theorem.
\begin{lemma}[{\cite[Lemma~5.5]{mimi}}]
 \label{lem:menger-two-paths}
 In a network $\Hbb$ with source $s$ and sink $t$, there are two $st$-paths $P_1,P_2$
 such that an edge~$e$ is a cutedge of $\Hbb$ if and only if both~$P_1$ and~$P_2$ use~$e$.
\end{lemma}

By \cite[Lemma~5.3]{mimi}, for every network $\Hbb$, there is an ordering $e_1,\ldots,e_p$ of its cutedges such that for every $st$-path $P$, the cutedges appear in $P$ in this order. Then $(e_1,\ldots,e_p)$ is called the \emph{cutedge sequence} of $\Hbb$.

Let $\Hbb$ be a network with hypergraph $H$, cutedge sequence $(e_1, \ldots, e_p)$, source~$s$ and sink~$t$,
and denote $e_0\eqdef \sg{s}$ and $e_{p+1}\eqdef \sg{t}$.
By \cite[Lemma~5.4]{mimi}, every component $C$ of $H-\sg{e_1,\ldots, e_{p}}\eqdef (V(H), E(H)\setminus\sg{e_1,\ldots, e_{p}})$ is of exactly one of the following two types:
\begin{enumerate}[label = (\arabic*)]
 \item\label{item: bridges} either $C$ intersects exactly two of the hyperedges $e_0,\ldots,e_{p+1}$, in which case there exists some $i\in \sg{0,\ldots,p}$ such that these two hyperedges are $e_i, e_{i+1}$.
   Such a component is called an \emph{$(e_i,e_{i+1})$-bridge}.
 \item\label{item: appendices} or $C$ intersects exactly one of the hyperedges $e_0,\ldots,e_{p+1}$, in which case it must be one of the cutedges $e_i$ for some $i \in\sg{1,\ldots,p}$.
   Such components are called \emph{$e_i$-appendices}.
\end{enumerate}

\paragraph{Thin networks}
Let us now generalise the key invariant used in the proof of~\cite[Lemma~5.2]{mimi}, called \emph{$k$-thinness} (cf.~\cite[Definition~5.6]{mimi}).
Let~$U$ be a fixed ground set, $\Vc$ be a family of subsets of~$U$, and~$\Hbb$ be a network with hypergraph~$H$, cutedge sequence $(e_1,\ldots,e_p)$, source~$s$ and sink~$t$. We also set $e_0\eqdef \sg{s}$ and $e_{p+1}\eqdef \sg{t}$. For each $i\in \sg{0,\ldots,k}$, let~$V_i$ be the union of the vertex sets of all $(e_i, e_{i+1})$-bridges, and for each $i\in \sg{1,\ldots,p}$, let~$W_i$ be the union of the vertex sets of all the $e_i$-appendices.

We say that $\Hbb$ is \emph{$(\Vc,k)$-thin} if $V(H)\subseteq U$,
all the hyperedges of $H$ have size at most $k$, and:
\begin{enumerate}[label=(\alph*)]
 \item\label{item: linearV} for each $i\in \sg{0,\ldots,p}$, the hypergraph $H[V_i]$ has a path-decomposition of adhesion at most $2k$, whose leftmost bag contains $V_i\cap e_i$, whose rightmost bag contains $V_i\cap e_{i+1}$, and whose bags all belong to $\Vc^{+k}$;
 \item\label{item: linearW} for each $i\in \sg{1,\ldots,p}$, the hypergraph $H[W_i]$ has a path-decomposition of adhesion at most $k$, whose leftmost bag contains $W_i\cap e_i$, and whose bags all belong to $\Vc$.
\end{enumerate}

The next lemma shows that thinness implies the path-decomposition condition of \cref{lem: fameux52-general}.

\begin{lemma}[{Generalisation of \cite[Lemma~5.7]{mimi}}]\label{lem: 57-generalisation}
  Consider a ground set~$U$, a family~$\Vc$ of subsets of $U$ such that $\emptyset\in \Vc$,
  and a $(\Vc,k)$-thin network~$\Hbb$. Then~$\Hbb$ has a path-decomposition of adhesion at most~$2k$ whose bags all belong to~$\Vc^{+k}$.
\end{lemma}
We omit the proof, which immediately generalises the one of \cite[Lemma~5.7]{mimi}, and simply consists in properly concatenating the path-decompositions of the hypergraphs $H[V_i]$ with the ones of the hypergraphs $H[W_i\cup e_i]$.
The hypothesis that $\emptyset\in \Vc$ is required in order to allow using any set of size at most~$k$ (and in particular any cutedge) as bag in the path-decomposition.

\paragraph{Substitution}
Let $H, K$ be hypergraphs and $e$ be an hyperedge of $H$, such that $V(H)\cap V(K)=e$.
Define $H[e\rightarrow K]$ to be the hypergraph with vertex set $V(K)\cup V(H)$, whose hyperedges are all hyperedges of~$H$ and~$K$, except~$e$ (if $e$ occurs multiple times in $E(H)\cup E(K)$, then we only remove one of its occurrences).
When~$\Hbb$ is a network with underlying hypergraph~$H$,
we also write $\Hbb[e \to K]$ for the network with hypergraph $H[e \to K]$ and with the same source and sink as~$\Hbb$.

The most technical part of the proof is the following lemma, showing that under appropriate conditions, substitution preserves thinness.
\begin{restatable}[{Generalisation of \cite[Lemma 5.8]{mimi}}]{lemma}{thinsubst}
 \label{lem: 58-general}
 Let~$\Vc$ be a downwards-closed family of subsets of a fixed ground set $U$.
 Let~$\Hbb$ be a $(\Vc,k)$-thin network with underlying hypergraph~$H$, and~$K$ a hypergraph such that~$V(K)\in\Vc$, whose hyperedges all have order at most $k$.
 Moreover, assume that there exists a cutedge~$e$ of~$\Hbb$ such that $V(H)\cap V(K)=e$.
 Then, $\Hbb[e\rightarrow K]$ is also $(\Vc,k)$-thin.
\end{restatable}
We will detail the proof of \cref{lem: 58-general} in \cref{sec:thinness-substitution},
as it requires some minor modifications compared to the bounded treewidth case.
For now, we sketch how to use it to prove \cref{lem: fameux52-general}.
The proof is essentially the same as the one of \cite[Lemma~5.2]{mimi}.

\begin{proof}[Proof of Lemma \ref{lem: fameux52-general}]
  Recall that we are given a sane tree-decomposition $(T,\beta)$ with adhesions of size~$k$, and two vertices~$u,v$ in its root bag.
  Further, $\Vc$ is the set of all subsets of the bags of~$(T,\beta)$.

  To any non-empty prefix $Z$ of $T$, one associates the network~$\Hbb_Z$ with underlying hypergraph $\hyp(T,Z)$, source~$u$, and sink~$v$.
  Note that~$u,v$ are indeed vertices of~$\Hbb_Z$ since the root of~$T$ belongs to~$Z$.
  The construction incrementally builds a non-empty prefix~$Z$ of~$T$,
  maintaining the invariant that~$\Hbb_Z$ is $(\Vc,k)$-thin.
  By \cref{lem: 57-generalisation}, this implies that $\hyp(T,Z)$ has a path-decomposition satisfying Condition~\eqref{item: hypertorso} in the statement.

  Initially, $Z \eqdef \sg{r}$ contains only the root~$r$ of~$T$.
  Thus the vertex set of $\hyp(T,Z)$ is $\beta(r)$, which is in~$\Vc$, so~$\Hbb_Z$ is trivially $(\Vc,k)$-thin.
  Given now any prefix~$Z$ satisfying the previous conditions,
  one can apply \cref{lem:menger-two-paths} to the associated network to find two $uv$-paths~$P_1,P_2$ such that the only edges used by both~$P_1$ and~$P_2$ are the cutedges.
  If no edge of $\mathfrak{e}(T,\bag)$ is used by both~$P_1$ and~$P_2$, then Condition~\eqref{item: 2paths} holds, and we are done. 

  We now assume that this is not the case, meaning that there is some hyperedge $e \in \mathfrak{e}(T,\bag)$ which is a cutedge of $\Hbb_Z$.
  By definition of $\mathfrak{e}(T,\bag)$, this hyperedge~$e$ corresponds to the adhesion of some node~$z\in \partial Z$.
  Consider then $Z' \eqdef Z \cup \sg{z}$, which is still a prefix of~$T$.
  Also, let~$K$ denote the hypertorso of~$z$ in $(T,\beta)$.
  One may then check that the network $\Hbb_{Z'}$ is exactly $\Hbb_Z[e\rightarrow K]$.
  It follows by \cref{lem: 58-general} that $\Hbb_{Z'}$ is also $(\Vc,k)$-thin, allowing to continue with~$Z'$.
  As $T$ is finite, this process must eventually halt with a prefix~$Z$ satisfying Condition~\eqref{item: 2paths}, proving the result.
\end{proof}

Finally, the proof of \cref{lem: fameux59-general} using \cref{lem: fameux52-general},
while highly non-trivial, is exactly the same as the proof of \cite[Lemma~5.9]{mimi}. We thus omit it.

\subsection{Substitutions in thin networks}
\label{sec:thinness-substitution}
This section proves \cref{lem: 58-general}, which we restate here.
\thinsubst*

The proof is similar to \cite[Lemma 5.8]{mimi}, the only significant difference being that~$K$ no longer has bounded size.
In some cases, their proof added all of~$K$ to all bags of some path-decomposition,
which for us would break the condition that bags are in~$\Hc^{+O(k)}$.
In these cases, we will instead add~$K$ in a new bag, and add some fixed-size subset of~$K$ to all other bags.

\begin{proof}
  In the network~$\Hbb$, denote by~$s,t$ the source and sink, and by $(e_1,\dots,e_p)$ the cutedge sequence, as well as $e_0 \eqdef \sg{s}$ and $e_{p+1} \eqdef \sg{t}$.
  Let $\ell\in \sg{1,\ldots,p}$ be the index of~$e$, i.e.\ $e=e_{\ell}$.
  Call $\hH\eqdef H[e\rightarrow K]$ the substitution we are interested in,
  and $\hHbb \eqdef \Hbb[e \rightarrow K]$ the corresponding network (i.e.\ the network with hypergraph~$\hH$, source~$s$, and sink~$t$).
  As in the definition of thinness, for each $i\in \sg{0,\ldots,p}$, let $V_i$ denote the union of the vertex sets of all $(e_i, e_{i+1})$-bridges in $\Hbb$, and for each $i\in \sg{1,\ldots,p}$, let $W_i$ denote the union of the vertex sets of all $e_i$-appendices in $\Hbb$.
 
 As shown in \cite{mimi}, the cutedge sequence of $\hHbb$ is of the form
 \[(e_1,\ldots, e_{\ell-1}, f_1,\ldots, f_q, e_{\ell+1},\ldots, e_{p}),\]
 where $f_1,\ldots, f_q$ are hyperedges of $K$ (with potentially $q=0$ if there is no such cutedge). We treat separately the special case $q=0$ at the end of the proof.
 The proof from \cite[Lemma 5.8]{mimi} gives that for any $i\in \sg{0, \ldots, p}\setminus \sg{\ell-1, \ell}$, the $(e_i,e_{i+1})$ bridges of $\hHbb$ are exactly the $(e_i, e_{i+1})$-bridges of~$\Hbb$,
 and that for each $i\in \sg{1,\ldots, p}\setminus \sg{\ell}$, the $e_i$-appendices of $\hHbb$ are exactly the $e_i$-appendices of~$\Hbb$. Furthermore, the only other possible cutedge components of $\hHbb$ are of the following types (see Figure \ref{fig: L58}):
 \begin{itemize}
  \item for each $j\in \sg{1,\ldots,q-1}$, all $(f_j, f_{j+1})$-bridges of $\hHbb$ are included in $W_{\ell}\cup V(K)$;
  \item for each $j\in \sg{1,\ldots,q}$, all $f_j$-appendices of $\hHbb$ are included in $W_{\ell}\cup V(K)$;
  \item all $(e_{\ell-1}, f_1)$-bridges are included in $V_{\ell-1}\cup W_{\ell}\cup V(K)$,
    and conversely all of~$V_{\ell-1}$ is covered by $(e_{\ell-1}, f_1)$-bridges;
  \item all $(f_q, e_{\ell+1})$-bridges are included in $V_{\ell}\cup W_{\ell}\cup V(K)$,
    and conversely all of~$V_\ell$ is covered by $(f_q,e_{\ell+1})$-bridges.
 \end{itemize}
 For each $j\in \sg{0, \ldots , q}$, let $X_j$ denote the union of the vertex sets of the 
 $(f_j, f_{j+1})$-bridges in $\hHbb$, where we set $f_0 \eqdef  e_{\ell-1}$, and $f_{q+1} \eqdef  e_{\ell+1}$, and for each $j\in \sg{1,\ldots, q}$, let $Y_j$ denote the union of the vertex sets of the $f_j$-appendices in~$\hHbb$ (see Figure \ref{fig: L58}).

 Let us now check that these sets satisfy the thinness conditions.
 First note that~$\hH$ has no hyperedge of order more than~$k$, since neither~$H$ nor~$K$ do.
 For any $i\in \sg{0,\ldots,p}\setminus \sg{\ell-1, \ell}$,
 since the union of all $(e_i,e_{i+1})$-bridges is the same in~$\Hbb$ and~$\hHbb$, and since $\Hbb$ is $(\Vc, k)$-thin, 
 the subhypergraph induced by this union also has a path-decomposition satisfying Property~\ref{item: linearV} of thinness.
 For the same reason, the union of $e_i$-appendices of $\hHbb$ for $i\in \sg{1,\ldots,p}\setminus \sg{\ell}$ has a path-decomposition with Property~\ref{item: linearW}.

 It remains to exhibit path-decompositions of the hypergraphs $\hH[X_j], \hH[Y_j]$ satisfying the properties from the definition of $(\Vc, k)$-thin hypergraphs.
 
\begin{figure}[htb]
  \centering  
  \includegraphics[scale=0.68]{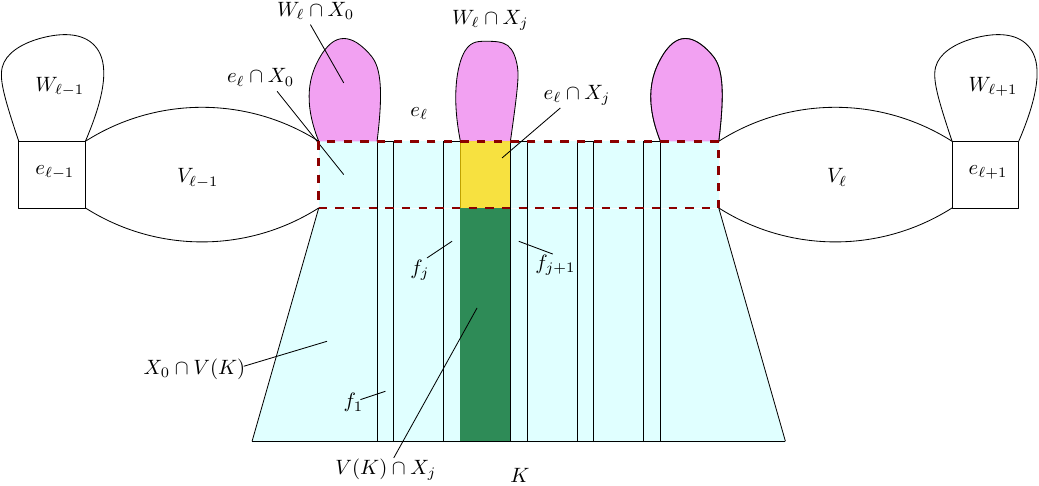}
  \caption{Configuration in the proof of Lemma \ref{lem: 58-general}. The dotted rectangle represents $e_{\ell}$, the region in light blue represents the vertices of $K$, and the purple regions represent $W_{\ell}$. Note that the hyperedges $f_j$ are not necessarily disjoint in general.}  
  \label{fig: L58}
\end{figure}

We start by considering some $j\in \sg{1,\ldots, q-1}$, and show that $\hH[X_j]$ has a path-decomposition satisfying \ref{item: linearV}.
Since~$\Hbb$ is $(\Vc,k)$-thin, there is a path-decomposition~$P$ of~$H[W_\ell]$ of adhesion at most~$k$,
whose leftmost bag contains $W_\ell \cap e_{\ell}$, and whose bags all belong to $\Vc$.
As $V(K)\cap V(H)=e_{\ell}$, the set $e_\ell$ separates~$V(K)$ from~$W_\ell$ in~$\hH$ (since $W_\ell \subseteq V(H)$), so we can add~$V(K)$ as a bag to the left of~$P$ to obtain a path-decomposition~$P'$ of $\hH[W_\ell \cup V(K)]$.
This first bag~$V(K)$ is in~$\Vc$ by assumption, and its adhesion is $e_\ell \cap W_\ell$, which has size at most~$k$.
Finally, since~$X_j$ is contained in~$W_\ell \cup V(K)$, we can restrict each bag of~$P'$ to its intersection with~$X_j$, yielding a path-decomposition~$P''$ of~$\hH[X_j]$, with adhesion at most~$k$ and bags in~$\Vc$ (we use here that~$\Vc$ is downwards-closed).
By construction, the leftmost bag of~$P''$ is $V(K) \cap X_j$, which contains~$f_j \cap X_j$.
Finally, we add $f_{j+1} \cap X_j$, which has size at most~$k$, to all bags of~$P''$.
The resulting decomposition of~$\hH[X_j]$ has adhesion at most~$2k$, bags in~$\Vc^{+k}$, and contains~$f_j \cap X_j$ in the leftmost bag and $f_{j+1} \cap X_j$ in the rightmost one, as required by \ref{item: linearV}.

Next we show that $\hH[X_0]$ (and symmetrically $\hH[X_{q}]$) admits a path-\break decomposition satisfying \ref{item: linearV}.
Recall that we have $X_0 \subseteq V_{\ell-1} \cup W_\ell \cup V(K)$, and~$V_{\ell-1} \subseteq X_0$.
Since~$\Hbb$ is $(\Vc,k)$-thin, there is a path-decomposition~$P_1$ of~$H[V_{\ell-1}]$ with adhesion at most~$2k$, bags in~$\Vc^{+k}$, and containing $e_{\ell-1} \cap V_{\ell-1}$ and $e_\ell \cap V_{\ell-1}$ in the leftmost and rightmost bags respectively.
This is also a decomposition of~$\hH[V_{\ell-1}]$, since the latter and~$H[V_{\ell-1}]$ can only differ by hyperedges contained in~$e_\ell$,
and $e_\ell \cap V_{\ell-1}$ appears in a bag of the decomposition.
Next, we consider the path-decomposition~$P'$ of $\hH[W_{\ell}\cup V(K)]$ with adhesion~$k$ and bags in~$\Vc$ from the previous paragraph, whose leftmost bag is~$V(K)$, and thus contains~$e_{\ell}$. We let $P_2$ be the path-decomposition of $\hH[W_{\ell}\cup V(K)]$ obtained after adding $f_1$ in every bag of $P'$. Then~$P_2$ as adhesion at most~$2k$, bags in~$\Vc^{+k}$,
its leftmost bag is~$V(K)$, and its rightmost bag contains~$f_1$.
Since the vertex set of~$e_\ell$ separates~$V_{\ell-1}$ from $W_\ell \cup V(K)$ in~$\hH$,
and the rightmost bag of~$P_1$ contains $V_{\ell-1} \cap e_\ell$, it follows that concatenating~$P_1$ with~$P_2$ yields a path-decomposition~$P$ of $\hH[V_{\ell-1} \cup W_\ell \cup V(K)]$.
Clearly, the bags of~$P$ are all in~$\Vc^{+k}$, and the leftmost and rightmost bags contain~$e_{\ell-1} \cap V_{\ell-1}$ and~$f_1$ respectively.
The adhesion between~$P_1$ and~$P_2$ is contained in~$e_\ell$, so all adhesions of~$P$ have size at most~$2k$.
Finally, if we restrict the decomposition~$P$ by keeping for each bag its intersection with~$X_0$, we obtain a path-decomposition that satisfies condition \ref{item: linearV}
(note that the fact that the leftmost bag contains~$e_{\ell-1} \cap X_0$ follows from the equality $e_{\ell-1} \cap X_0 = e_{\ell-1} \cap V_{\ell-1}$).

Finally, for $j\in \sg{1,\ldots, q}$, we show that $\hH[Y_j]$ has a path-decomposition satisfying \ref{item: linearW}.
Once again, consider the path-decomposition~$P'$ of $\hH[V(K)\cup W_\ell]$ defined two paragraphs ago, with adhesion~$k$, bags in~$\Vc$, and the leftmost bag being~$V(K)$, which in particular contains~$f_j$.
Restricting each bag of~$P'$ to its intersection with~$Y_j$ yields a path-decomposition of~$\hH[Y_j]$ satisfying~\ref{item: linearW}.

Combining the previous three paragraphs, we conclude that when $q>0$, then~$\hHbb$ is indeed $(\Vc,k)$-thin.
It remains to prove that it is also the case when $q=0$. For this special case, we can conclude in a similar way as in previous paragraph: assume that $q=0$, and let $X_0$ denote the union of all $(e_{\ell-1},e_{\ell+1})$-bridges in $\hHbb$.
As $\Hbb$ is $(\Vc,k)$-thin, we only have to show that $\hH[X_0]$ has a path-decomposition satisfying \ref{item: linearV}. As $\Hbb$ is $(\Vc,k)$-thin, there exist respectively some path-decompositions $P_1, P_2, P_3$ of $H[V_{\ell-1}], H[W_{\ell}]$ and $H[V_{\ell}]$, such that $P_1,P_3$ satisfy \ref{item: linearV}, and $P_2$ satisfies \ref{item: linearW}. We let $P'_2$ be the path-decomposition obtained from $P_2$ after adding the vertices of $e_{\ell}$ in all bags, and now consider the path-decomposition of $\hH[X_0]$ obtained by concatenating $P_1$, then $P'_2$, then one bag equal to $V(K)$, and finally $P_3$. As $V(K)$ belongs to $\Vc$, and $e_{\ell}=V(K)\cap V(H)$, this path-decomposition indeed satisfies \ref{item: linearV}.
\end{proof}
This completes the proof of \cref{lem: 58-general}, and with it \cref{thm: main-td}.

\section{Graphs excluding a minor}
\label{sec: minors}
In this section, we prove our main result that graphs excluding a fixed graph as a minor have bounded basis number.
\minorbn*
Using \cite{miraftab2026pathwidth} to obtain polynomial bounds in \cref{thm: main-td},
as well as the recent polynomial bounds for the Graph Minor Structure Theorem of Gorsky, Seweryn and Wiederrecht \cite{GSW25}, we obtain that the bound~$f_{\ref{thm: main-minor}}$ is a polynomial
$$f_{\ref{thm: main-minor}}(t)=O(t^c)$$ for some constant $c\leq 32210$.

\paragraph{Overview}
The proof of Theorem \ref{thm: main-minor} consists of a combination of the Graph Minor Structure Theorem of Robertson and Seymour \cite{RSXVI} (see Theorem \ref{thm: RS-structure}) together with Theorem \ref{thm: main-td}. For every fixed graph $H$, the Graph Minor Structure Theorem allows to find a tree-decomposition of any $H$-minor free graph $G$ with bounded adhesion, whose torsos are \emph{almost-embeddable} in some surface of bounded genus (see Subsection \ref{sec: almost-embeddable} for definitions).
Since almost-embeddability is a monotone property,
by \cref{thm: main-td}, we only need to prove that these almost-embeddable graphs have bounded basis number; this is \cref{thm: bn-almost}.

Let us sketch a false but enlightening proof of this fact.
Roughly speaking, a graph~$G$ is almost-embeddable in a surface~$\Sbb$ if it can be constructed as follows.
First take a graph~$G_0$ embedded in~$\Sbb$.
Pick a bounded number of faces~$F_1,\dots,F_k$ of~$G_0$,
and for each of them add a \emph{vortex}~$G_i$, i.e.\ a graph with bounded pathwidth whose intersection with~$G_0$ is exactly the cycle boundary of~$F_i$.
Finally, add a bounded number of \emph{apices}, i.e.\ vertices with arbitrary edges to the rest of the graph.
The base graph~$G_0$ has bounded basis number by \cref{thm:bn-surface}.
Since the vortices~$G_i$ have bounded pathwidth, they have bounded basis number by \cref{thm:pathwidth-linear},
and since~$G_i$ is glued to~$G_0$ on the cycle boundary of~$F_i$, \cref{lem:bn-connected-separator} shows that $G_0 \cup G_1 \cup \dots \cup G_k$ still has bounded basis number.
Finally, adding apices preserves bounded basis number by \cref{lem:bn-edits}, proving that~$\bn(G)$ is bounded.

The mistake here is to implicitly assume that the embedding of~$G_0$ is \emph{cellular} (i.e.\ that each face is homeomorphic to a disc), and thus that its face boundaries are cycles, or even connected.
When~$G_0$ is connected, one can always find a cellular embedding of~$G_0$ on a surface with genus no larger than that of~$\Sbb$,
but the faces of this new embedding may be different, and might not allow constructing the desired vortices. Moreover, $G_0$ might not even be connected.
It is also always possible to replace~$G$ and~$G_0$ by some supergraphs~$G'$ and~$G'_0$ in the construction so that the desired embedding of~$G'_0$ is cellular,
but since basis number is not monotone under taking subgraphs, proving that~$\bn(G')$ is bounded is not sufficient.
We thus need the more formal and general definition of almost-embeddability,
allowing a non-cellular embedding of~$G_0$, and vortices in faces that are not homeomorphic to disks, see \cref{sec: almost-embeddable}.
In particular, the subgraphs~$G_0 \cap G_i$ on which vortices are glued are not necessarily connected in this definition, preventing us from applying \cref{lem:bn-connected-separator}.
This is the main source of technicalities in the proof of Theorem~\ref{thm: bn-almost}.
To overcome this issue, we will use in particular the fact that any graph embedded in a surface of bounded genus whose dual has bounded diameter has bounded treewidth~\cite{Eppstein,Mazoit},
which we observe still holds for non-cellular embeddings (see \cref{coro: tw-dual}).

We give in Subsection \ref{sec: surfaces} basic definitions and results concerning graphs embedded in surfaces, as well as the aforementioned Corollary \ref{coro: tw-dual}. Subsection~\ref{sec: almost-embeddable} then presents the definition of almost-embeddability, and proves Theorem \ref{thm: bn-almost}. Finally, we conclude the proof of \cref{thm: main-minor} in Subsection~\ref{sec: proof-minor}, combining the Graph Minor Structure Theorem together with \cref{thm: main-td,thm: bn-almost}.

\subsection{Graphs embedded in surfaces}
\label{sec: surfaces}
We gather in this section some known results about graphs embeddable in surfaces of bounded genus, which we will need later in our proof of Theorem \ref{thm: main-minor}, when dealing with graphs almost embeddable in surfaces of bounded genus.

We allow graphs in this section to have loops and multi-edges (with finite multiplicities).
Note that removing multi-edges or loops to make a graph simple changes neither its treewidth, its diameter, nor its embeddability in a given surface.

\subsubsection{Topological embeddings in surfaces}
We briefly introduce here some basic concepts related to graphs embedded in surfaces, which will be used in our proofs, and refer to \cite{MoharThomassen} for a more comprehensive treatment of the topic.
We work in compact arc-connected surfaces~$\Sbb$ without boundary.
A simple \emph{curve} on~$\Sbb$ is an injective continuous mapping $\gamma: [0,1]\to \Sbb$.
The curve~$\gamma$ and its image $\gamma([0,1])$ in~$\Sbb$ are implicitly identified.
An open (resp.\ closed) \emph{disk} in~$\Sbb$ is a subspace homeomorphic to the open (resp.\ closed) unit disk in~$\mathbb R^2$.

Given a fixed surface $\mathbb S$, an \emph{embedding} of a graph $G$ in $\mathbb S$ is an injective mapping $\varphi$ sending vertices of $G$ to points of $\mathbb S$, and edges of $G$ to simple curves in $\mathbb S$ whose extremities are the images of the endpoints of each edge. In particular, the (topological) interiors of the images of edges must be pairwise disjoint, and also disjoint from the images of the vertices.
When we are given a pair $(G,\varphi)$, where $G$ is a graph and $\varphi$ is an embedding of $G$ in $\mathbb S$, we say that~$G$ is \emph{embedded} in $\mathbb S$.
We usually leave the embedding~$\varphi$ implicit, and whenever a graph is embedded in $\mathbb S$, we identify it with its image $\varphi(V(G)\cup E(G))$ in $\Sbb$.

If~$G$ is a (finite) graph embedded via~$\varphi$ in~$\Sbb$, then $\Sbb \setminus G$ is a finite collection of open disjoint arc-connected components, called the \emph{faces} of~$G$.
The topological boundary of a face~$F$ is denoted by~$\partial F$.
The boundary~$\partial F$ coincides with the image~$\varphi(H)$ of some subgraph $H \subseteq G$, and we also use~$\partial F$ to denote this subgraph.
The edges and vertices in~$\partial F$ are said to be \emph{incident} to~$F$.
A graph embedding in which all faces are open disks is called a \emph{cellular embedding}.
We stress out that we will need to consider non-cellular embeddings too.

Let~$G$ be embedded in~$\Sbb$ via~$\varphi$, and~$H$ a subgraph of~$G$ implicitly embedded by the restriction of~$\varphi$.
Then every face of $G$ is contained in some face of $H$.
If an embedded graph has at least one edge, then all its face boundaries contain at least one edge.

\subsubsection{Facial walks}
\label{sec: facial}
This subsection is dedicated to the proofs of Lemmas \ref{lem: boundaries-components} and \ref{lem: connecting-path-surface} below, which are very basic results on graphs embedded in surfaces. To prove them, we will first need some more notions about facial walks on surfaces, and use for this the formalism from \cite{MoharThomassen}.

\paragraph{Combinatorial embeddings}
A central concept introduced by Mohar and Thomassen \cite[Section 4]{MoharThomassen} when working on graphs in surfaces is the notion of \emph{combinatorial embedding}, which allows to describe topological embeddings of graphs on a surface in a purely combinatorial way, up to homeomorphism. We stress out that \cite{MoharThomassen} focuses on cellular embeddings of graphs on surfaces, while our proof will require to consider general embeddings. However, the notions and results we will use here also hold in this settings, and we refer to \cite{HR} for a more complete and formal treatment of general embeddings.

A \emph{combinatorial embedding} of a graph $G$ is a pair~$(\Pi, \varepsilon)$, where:
\begin{itemize}
 \item $\Pi=(\pi_{v})_{v\in V(G)}$ is a \emph{rotation system}, i.e.\  for each vertex $v\in V(G)$, $\pi_v:~E(G)\to E(G)$ is a cyclic permutation of the edges from $E(G)$ incident to $v$;
 \item $\varepsilon: E(G)\to \sg{-1,+1}$ is a \emph{signature} mapping.
\end{itemize}
For each vertex $v\in V(G)$ and edge $e\in E(G)$ incident to $v$, we call the cyclic sequence $(e,\pi_v(e),\pi_v^2(e),\ldots)$ the \emph{clockwise ordering} around $v$ in $(\Pi, \varepsilon)$,
and its inverse sequence the \emph{anticlockwise ordering} around $v$ in $(\Pi, \varepsilon)$.
Intuitively, if~$G$ is a graph embedded in a surface $\mathbb S$ and $v$ a vertex of $G$, then the circular ordering $\pi_v$ correspond to the clockwise/anticlockwise ordering induced in a small neighbourhood of $v$ on $\Sbb$ on the edges incident to $v$,
and for each edge $e\in E(G)$, the sign~$\varepsilon(e)$ indicates whether when travelling through the edge~$e$ on~$\Sbb$, the ``left and right'' sides around $e$ are twisted or not.
Each graph $G$ embedded in a surface~$\Sbb$ induces a natural canonical combinatorial embedding, which is unique, up to performing the following operation an arbitrary number of times: choose a vertex $v\in V(G)$, replace $\pi_v$ by $\pi_v^{-1}$ and all signatures of edges $e$ incident to $v$ by their inverse $-\varepsilon(e)$.

\paragraph{Facial walks}
Given a combinatorial embedding $(\Pi, \varepsilon)$ as above of a graph $G$, the \emph{face traversal procedure} is a way to produce (closed) walks starting from a vertex $v$ as follows.
For this definition, we consider edges as bi-oriented, i.e. if $e=uv\in E(G)$, the edge $e'=vu$ is also an edge of $E(G)$, distinct from $e$.
We always start the procedure in the ``clockwise mode''. Choose an initial edge $e=vu$ incident to $v$, and traverse $e$ from $v$ to $u$.
If $\varepsilon(e)=1$, then choose~$\pi_u(e)$ to be the next edge to traverse starting from $u$. If $\varepsilon(e)=-1$, then switch to the ``anticlockwise mode'' and choose $\pi_u^{-1}(e)$ to be the next edge to traverse starting from $u$.
Repeat then this procedure with the rule that traversing any edge~$e$ with signature $\varepsilon(e)=-1$ switches the mode between clockwise and anti-clockwise,
and the next edge to be traversed after~$e$ is the successor of~$e$ in the clockwise ordering~$\pi_v$ or the anticlockwise ordering~$\pi^{-1}_v$, depending on the mode.
The procedure stops when the initial configuration is reproduced, i.e.\ as soon as we traverse a second times the initial edge $e$, in the same direction and in clockwise mode.
This procedure is easily reversible, hence we never repeat twice the same configuration (i.e.\ we never traverse a given oriented edge in the same mode twice) before looping back to the initial configuration.

A \emph{facial walk} in $(\Pi, \varepsilon)$ is a walk $L=e_1\ldots e_k$ obtained after starting a face traversal from an initial edge $e_1$ (recall that the edges $e_i$ are oriented here). As facial walks are by definition of the procedure closed walks, we consider them up to inverse, and cyclic rotations, i.e.\ for every $m\geq 1$, $L$ is the same walk than the walks $e_k\ldots e_1$ and $e_{1+m}\ldots e_{k+m}$ (where indices are taken modulo $k$).
We also follow the convention that an isolated vertex in the graph gives rise to a trivial facial walk.

The crucial property that we will use is that if $G$ is a graph embedded on a surface $\Sbb$, and if $(\Pi, \varepsilon)$ denotes its associated canonical embedding, then every face boundary (in the topological sense) of $G$ is a disjoint union of facial walks (with respect to the above definition).
In the special case of cellular embeddings, the boundary of a face is a single facial walk instead (and in particular is connected).
Moreover, every edge $e$ of $G$, when considered as unoriented, appears twice in the collection of all facial walks of~$G$,
corresponding to the two sides of~$e$.
Precisely, one of the following three cases occurs:
\begin{itemize}
   \item Either $e$ appears exactly once in two face boundaries $\partial F, \partial F'$ for two distinct faces $F,F'$,
   \item or $e$ appears in two disjoint facial walks inside the boundary~$\partial F$ of a unique face~$F$,
   \item or $e$ appears twice (but in different orientations) within a single facial walk of a unique face~$F$.
     In this last case, $e$ is called \emph{singular}.
\end{itemize}
With this correspondence between facial walks and face boundaries in mind, the proofs of the next two lemmas will follow easily.

\begin{lemma}
 \label{lem: boundaries-components}
 Let $G$ be a graph embedded in a surface $\mathbb S$ with connected components $G_1,\ldots, G_m$.
 Then the set of facial walks in~$G$ is exactly the union of sets of facial walks of the embedded subgraphs~$G_i$.
\end{lemma}

\begin{proof}
  For a given topological embedding of~$G$, the corresponding combinatorial embedding $(\Pi,\varepsilon)$ of~$G$
  is the disjoint union $\Pi=\biguplus_{i=1}^m \Pi_i$ and $\varepsilon=\biguplus_{i=1}^m \varepsilon_i$ of the corresponding sub-embeddings $(\Pi_i, \varepsilon_i)$ of $G_i$ for each $i \in [m]$.
  The statement then easily follows from the combinatorial description of facial walks.
\end{proof}

\begin{lemma} \label{lem: connecting-path-surface}
  Consider a graph~$G$ embedded in~$\Sbb$, a face~$F$ of~$G$, and two distinct facial walks~$W_1,W_2$ in $\partial F$.
  If $\gamma$ is a simple curve connecting a vertex $x\in V(W_1)$ to a vertex $y\in V(W_2)$, such that the topological interior of $\gamma$ is included in $F$, then $\gamma$ does not disconnect $F$.
  This implies that each facial walk of~$G \cup \gamma$ contains as subset a facial walk of~$G$.
\end{lemma}

\begin{figure}[htb]
  \centering  
  \includegraphics[scale=0.6]{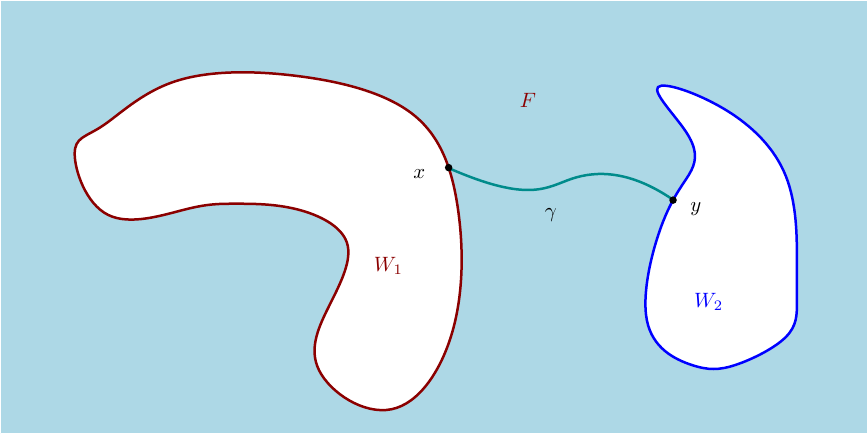}
  \caption{Configuration of Lemma \ref{lem: connecting-path-surface}.}  
  \label{fig: connecting}
\end{figure}

\begin{proof}
We set $e\eqdef xy$ and consider the graph  $G'\eqdef (V(G),E(G)\cup \sg{e})$. We consider the embedding of $G'$ on $\Sbb$, obtained from the one of $G$, and where we map $e$ to $\gamma$, and let $(\Pi', \varepsilon')$ be a some associated canonical combinatorial embedding of $G'$.
We claim that~$e$ is a singular edge in $G'$.

Indeed, let $W=e_1\ldots e_k$ be the closed facial walk in~$G'$ obtained when starting by traversing $e$ from $x$ to $y$ (i.e.\ $e_1=e$).
Suppose for a contradiction~$e$ is not singular, i.e.\ appears only once in~$W$.
Then one may check the subwalk $W'\eqdef e_2\ldots e_k$ is also a subwalk of a facial walk from $\partial F$ in $G$ connecting~$y$ to~$x$,
contradicting the assumption that~$x,y$ belong to distinct components of~$\partial F$.

Thus~$e$ is singular, meaning that its two sides still belong to the same face~$F$, i.e.\ adding~$e$ did not disconnect~$F$.
Finally, since the facial walk~$W$ in~$G'$ contains both occurrences of~$e$,
it is simple to check that any facial walk of~$G'$ other than~$W$ is also a facial walk of~$G$.
Since~$W$ intersects the facial walks~$W_1$ and~$W_2$ of~$G$, this implies that~$W_1$ and~$W_2$ are in fact fully contained in~$W$.
Combining these two facts, each facial walk of~$G'$ contains a facial walk of~$G$, as claimed.
\end{proof}

\subsubsection{Dual and diameter}
\label{sec: dual}
Given a graph $G$ embedded in a surface $\mathbb S$, its \emph{dual} is the graph $G^{\ast}$ whose vertices are the faces of $G$, and where for every edge $e\in E(G)$, if $F,F'$ denote the at most two faces of $G$ incident to $e$, 
we add an edge $e^{\ast}$ between $F$ and $F'$.
The graph $G^{\ast}$ also admits a natural embedding in $\mathbb S$. Since $\mathbb S$ is arc-connected, $G^{\ast}$ is always connected.
Note that while duals are usually considered only for cellular embeddings, the definition remains meaningful for non-cellular embeddings.
In this general setting, we may however have $(G^\ast)^\ast \neq G$.

One can however extend a graph to obtain a cellular embedding without changing the dual graph.
We would like to thank Louis Esperet for this suggestion.
\begin{lemma}
 \label{lem: cellularising}
 Let~$G$ be embedded via~$\varphi$ on a surface~$\Sbb$ of genus~$g$.
 Then there is a supergraph $G'$ of $G$, and a \emph{cellular} embedding $\varphi': G'\to \Sbb'$ for some surface~$\Sbb'$ of genus at most $g$, such that the dual of $G'$ with respect to $\varphi'$ is isomorphic to the dual of $G$ with respect to $\varphi$, plus some loops.
\end{lemma}
\begin{proof}
  Consider a face~$F$ of~$G$ with facial walks $W_1,\dots,W_m$.
  In~$\Sbb$, one can always replace the face~$F$ by a copy of the $2$-dimensional sphere with~$m$ holes, where the holes are glued to the walks $W_1,\ldots,W_m$ respectively --- in general~$F$ will be this sphere with holes, plus possibly some handles or crosscaps added.
  Call~$\Sbb'$ and~$\varphi'$ the surface and embedding of~$G$ obtained after performing this replacement for all faces.
  Then~$\Sbb'$ has genus at most~$g$, and the dual of~$G$ with respect to~$\varphi$ and to~$\varphi'$ are isomorphic.

  Assuming now that all faces are homeomorphic to spheres with holes, we construct~$G'$ as follows: for each face $F$ of $G$ with respect to $\varphi'$ whose boundary is not connected, we draw in $F$ a forest $T_F$ such that $F\setminus T_F$ is arcwise connected, that the union of $T_F$ with the facial walks of $F$ is connected, 
  and such that the leaves of $T_F$ belong pairwise to distinct components of $\partial F$. Such a forest can be constructed after iteratively applying Lemma \ref{lem: connecting-path-surface}.
  After adding this forest to~$G'$, the face~$F$ becomes homeomorphic to a disk,
  and the dual only changes in that a loop is added to~$F$ for each edge of~$T_F$.
  Repeating this for all faces, we obtain~$G'$ embedded cellularly in~$\Sbb'$, whose dual is isomorphic to the dual of~$G$ except for the additional loops.
\end{proof}

We can now state the two important results we will use on treewidth of graphs embedded in surfaces.
The first one was proved by Eppstein, based on Baker's technique, and states that graphs with bounded diameter and bounded genus have bounded treewidth.
\begin{theorem}[\cite{Eppstein}]
 \label{thm: diam-dual}
 If $G$ is a graph embedded in a surface of genus~$g$, then
 \[ \tw(G)=O(g\cdot\diam(G)). \]
\end{theorem}

The second is due to Mazoit, and relates the treewidth of a graph and its dual, in the case of cellular embeddings.
\begin{theorem}[\cite{Mazoit}]
 \label{thm: mazoit}
 Let $G$ be a graph cellularly embedded in a surface $\Sbb$ of genus at most $g$. Then
 $$\tw(G^{\ast})\leq \tw(G)+g+1.$$
\end{theorem}

Combining the previous three results, we obtain the following statement, which will be used in the next section.
\begin{corollary}
 \label{coro: tw-dual}
 Let~$G$ be a graph embedded in a surface $\mathbb S$ of genus $g$. Then
 \[ \tw(G)= O(g\cdot \diam(G^{\ast})). \]
\end{corollary}
\begin{proof}
  Let $G'$ and $\Sbb'$ be given by Lemma \ref{lem: cellularising}. Then
  \[ \tw(G)\leq \tw(G')
           \leq \tw(G'^{\ast}) + g +1        =O(g\cdot\diam(G'^{\ast}))
         =O(g\cdot\diam(G^{\ast})), \]
  where the first inequality holds as~$G$ is a subgraph of~$G'$,
  and the second and third follow from Theorems~\ref{thm: mazoit} and \ref{thm: diam-dual} respectively.
\end{proof}

\subsection{Almost embeddability}
\label{sec: almost-embeddable}
In this section, we give a formal definition of the notion of \emph{almost embeddability in a surface}, which plays a central role in the Graph Minor series of papers. Our main result is Theorem \ref{thm: bn-almost}, which states that graphs almost embeddable in surfaces of bounded genus have bounded basis-number. Our proof of \cref{thm: main-minor} will then immediately follows from a combination of this result with the Graph Minor Structure Theorem (see Theorem \ref{thm: RS-structure}) and Theorem \ref{thm: main-td}.
 
\paragraph{Apices}
We recall the definition of apices, which was already given in \cref{sec:hypergraph-defs}.
For class of graphs $\Gc$ and every $k\geq 0$, we let $\Gc^{+k}$ denote the class of those graphs $G$ for which there exists a subset $A$ of vertices (called \emph{apices}) of size at most $k$, such that $G-A\in \Gc$.

\paragraph{Vortices}
Let $G_0$ be a graph embedded in a surface $\Sbb$, let $F$ be a face of $G_0$, and $D$ a closed disk contained in~$\overline{F}$, such that~$D$ and~$G_0$ only intersect in points of $\partial D \cap V(G_0)$.
Denote by $U \eqdef \partial D \cap V(G_0)$ these intersection points,
and order them as $U \eqdef \{u_1, \ldots, u_n\}$ following one of the two cyclic orderings~$\preceq_D$ induced by~$\partial D$.
A \emph{vortex} attached to~$U$ is any graph~$G_U$ that intersects~$G_0$ only in~$U$.
There are many ways to define the \emph{width} of a vortex that are equivalent up to constant factors, and we choose here to use the one related to pathwidth:
the width of the vortex~$G_U$ is the smallest~$k$ such that~$G_U$ has a path-decomposition of width $k$, with bags $B_1,\dots,B_n$ such that $u_i \in B_i$ for each $i\in [n]$,
with again~$u_1,\dots,u_n$ the vertices of~$U$ ordered cyclically along~$D$ (the choice of the first vertex~$u_1$ may affect the width with this definition, but only by a multiplicative factor of at most~2).

Let $U_1,\ldots, U_k$ be subsets of $V(G_0)$, and $D_1,\ldots, D_k$ be pairwise disjoint closed disks on $\Sbb$ such that for each 
$i\in [k]$, there exists a face $F_i$ of $(G_0, \varphi)$, such that $D_i\subseteq \overline{F_i}$, and $\partial D_i \cap V(G)=U_i$.
If $G_{U_1},\ldots, G_{U_k}$ are vortices such that each $G_{U_i}$ attaches to $U_i$, then we similarly say that the graph $G \eqdef G_0\cup G_{U_1}\cup \cdots\cup G_{U_k}$ is obtained 
after attaching to $G_0$ the vortices $G_{U_1}, \ldots, G_{U_k}$.

\paragraph{Almost embeddability}
Let $\mathbb S$ be a surface and $a,k\in \mathbb N$. A graph $G$ is \emph{$(a,k)$-almost embeddable} in $\mathbb S$ if $G$ if there exists some subset $A\subseteq V(G)$ of vertices (apices) of size at most $a$, and some subgraphs $G_0, \ldots, G_k$ such that $V(G)\setminus A=G_0\cup \cdots \cup G_k$, and such that:
\begin{itemize}
 \item[$(a)$] $G_0$ admits an embedding $\varphi$ in $\mathbb S$;
 \item[$(b)$] the graphs $G_1,\ldots,G_k$ are pairwise vertex disjoint;
 \item[$(c)$] for each $i\in [k]$, there exists a face $F_i$ of $(G_0, \varphi)$, and a closed disk $D_i\subseteq \overline{F_i}$ such that the disks $D_i$ are pairwise disjoint,
   each disk~$D_i$ intersects~$G$ only in $U_i\eqdef \partial D_i\cap V(G_0)$,
   and~$G_i$ is a vortex of width at most~$k$ attached to~$U_i$.
\end{itemize}
 
For each $a,k,g\in \mathbb N$, we let $\mathcal G_{a,k,g}$ denote the class of graphs which are $(a,k)$-almost embeddable in a surface of genus at most $g$. Note that $\mathcal G_{a,k,g}$ is monotone. 
The Graph Minor Structure Theorem (see Theorem \ref{thm: RS-structure} later) states that for every fixed graph~$H$, there exist $a,k,g$ such that every graph $G$ excluding $H$ as a minor can be decomposed using the graphs of $\Gc_{a,k,g}$ as basic bricks.
We prove that~$\Gc_{a,k,g}$ has bounded basis number:
\begin{theorem}
 \label{thm: bn-almost}
 Let $a,k,g\in \mathbb N$ and $G\in \Gc_{a,k,g}$. Then 
 $$\bn(G)= f_{\ref{thm:bn-treewidth}}(\alpha_{\ref{coro: tw-dual}}\cdot (g+1) \cdot (6k^2+5k-1)) + O(\log^2g)+2a,$$
 where $\alpha_{\ref{coro: tw-dual}}$ denotes the multiplicative constant from Corollary \ref{coro: tw-dual}.
\end{theorem}
The rest of this section is dedicated to the proof of this result.

Firstly, by \cref{lem:bn-edits}, adding~$a$ apices increases the basis number by at most~$2a$. We can thus focus on the case $a=0$.
Throughout this section, we fix a graph $G \in \Gc_{0,k,g}$, obtained from a graph~$G_0$ embedded in a surface~$\Sbb$ of genus at most~$g$
after adding vortices~$G_1,\dots,G_k$ of width at most~$k$.
Furthermore, we denote by~$F_i,D_i$ the face and disk containing~$G_i$, and $U_i \eqdef V(G_0) \cap V(G_i)$  for each $i\in [k]$.
We will also assume that $G_0$ has at least one edge, as otherwise $G$ consists in the disjoint union of the vortices $G_i$. As each of the graphs $G_i$ has pathwidth at most $k$, \cref{thm:pathwidth-linear} then implies that in this special case, $\bn(G)= O(k)$.

The first step of the proof is to apply Lemma \ref{lem:bn-almost-connected-separator} for two subgraphs~$G_X,G_Y$ with $G = G_X \cup G_Y$,
where $G_X \eqdef G_0$, and~$G_Y$ is a supergraph of $G_1 \cup \dots \cup G_k$ that we will now define.

Let $\mathcal C_0$ denote the set of connected components of~$G_0$, seen as vertex subsets.
For each component $W \in \Cc_0$, and each face~$F_i$ containing a vortex,
we define the (possibly empty) subgraph $\Theta_{W,i} \eqdef G_0[W] \cap \partial F_i$,
i.e.\ the union of those facial walks of~$F_i$ that are contained in~$W$.
Note that $\Theta_{W,i}$ is not necessarily connected, and that  $U_i\cap W\subseteq V(\Theta_{W,i})$ for each $i\in [k]$ and $W\in \Cc_0$.
Next, taking the union over all vortex faces, we define $\Theta_W \eqdef \bigcup_{i\in [k]} \Theta_{W,i}$ for each $W\in \Cc_0$.
Choose now a connecting subforest~$T_W$ in~$G_0[W]$ for the family of connected components of~$\Theta_W$, given by \cref{rem: connecting-forest} (with~$T_W$ being possibly empty, when~$\Theta_W$ is already connected).
Finally, we define $\Gamma_W \eqdef \Theta_W \cup T_W$, which is a connected subgraph of~$G_0[W]$, and
$$\Gamma \eqdef \bigcup_{W\in \Cc_0}\Gamma_W.$$
See Figure \ref{fig: Gamma} for an illustration.
Clearly, $\Gamma$ is an embedded subgraph of $G_0$, whose connected components are the $\Gamma_W$'s.
Observe that by construction, for each $i\in [k]$, the boundary $\partial F_i$, when considered as a subgraph of $G_0$, is a subgraph of $\Gamma$. In particular $U_i \subseteq V(\Gamma)$.

\begin{figure}
  \centering  
  \includegraphics[scale=0.7]{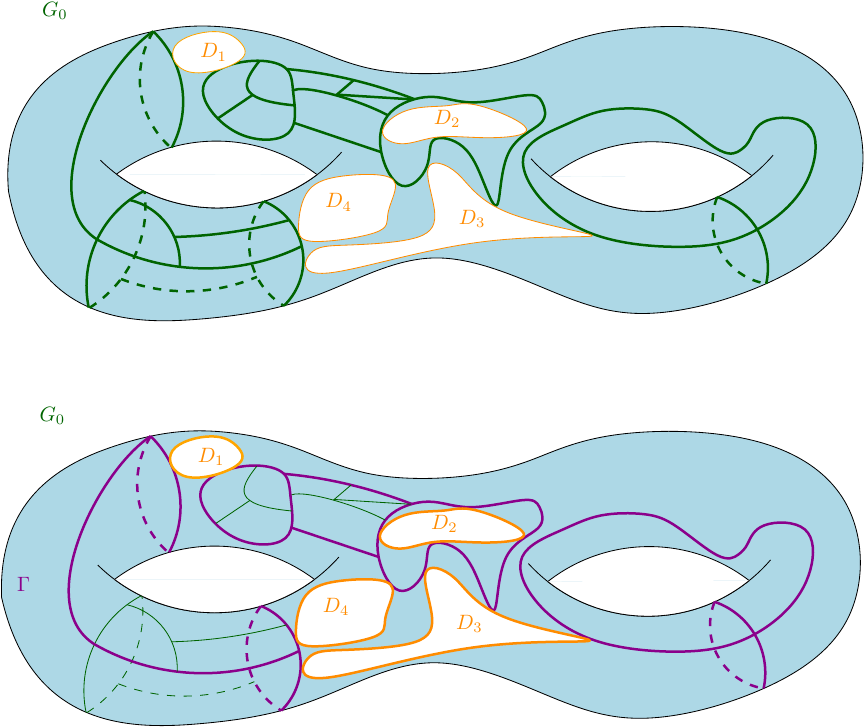}
  \caption{Top: In dark green, a graph $G_0$ embedded in the orientable surface of genus $2$. In orange a collection of pairwise disjoint closed disks which only intersect $G_0$ in $V(G_0)$. Note that $D_1,D_3$ and $D_4$ are contained in the same face of~$G_0$.
  Bottom: The graph $\Gamma$ constructed with respect to the disks $D_1,D_2,D_3,D_4$ is represented in thick, magenta. The graph $H_\Gamma$ defined later in the proof corresponds to the union of the magenta graph and the orange cycles.}
  \label{fig: Gamma}
\end{figure}

We call a connected component, respectively a facial walk \emph{trivial} if it consists only of a single vertex.
\begin{claim}\label{clm: Gamma-W}
  For any $W \in \Cc_0$, any non-trivial facial walk of~$\Gamma_W$ contains an edge incident to one of the faces $F_1,\ldots,F_k$.
\end{claim}
\begin{claimproof}
  First note that~$G_0[W]$ is non-trivial, since we consider a non-trivial walk in it.
  It follows that any facial walk of~$G_0[W]$ is non-trivial,
  which is in particular the case of the facial walks constituting~$\Theta_W$.
  Thus~$\Theta_W$ has no isolated vertices, and by definition consists solely of edges incident to $F_1,\dots,F_k$.

  Observe now that the connecting forest~$T_W$ is obtained by starting from~$\Theta_W$, choosing iteratively two distinct components of a subgraph $H$ of $\Gamma$, and connecting them with a path of $G_0[W]$, whose associated simple curve $\gamma$ in $\mathbb S$ is internally disjoint from $H$.
  By repeated applications of \cref{lem: connecting-path-surface},
  it follows that each facial walk of~$\Gamma_W$ contains as a subset a facial walk of~$\Theta_W$, which itself contains an edge incident to one of $F_1,\dots,F_k$, proving the result.
\end{claimproof}

\begin{claim}
  \label{clm: union-Gamma-W}
  For each face $F$ of $\Gamma$, there is an edge in the face boundary $\partial F$ which is incident to one of the faces $F_1,\ldots, F_k$.
\end{claim}
\begin{claimproof}
  Recall that we assumed~$G_0$ to be non-trivial, hence the face boundary~$\partial F$ must contain some edge from some non-trivial component $W \in \Cc_0$.
  In particular, by \cref{lem: boundaries-components}, $\partial F$ also contains a non-trivial facial walk of $\Gamma_W$.
  The result follows by \cref{clm: Gamma-W}.
\end{claimproof}

We now set $G_X \eqdef G_0$, and $G_Y \eqdef \Gamma \cup \bigcup_{i=1}^k G_i$, and claim that~$(G_X,G_Y)$ satisfies the assumptions of \cref{lem:bn-almost-connected-separator}.
\begin{claim}
 \label{clm: almost-embedd-separation}
 For each connected component~$W$ of~$G_X$, the graph $G_X[W] \cap G_Y$ is connected.
\end{claim}
\begin{claimproof}
  Let~$W$ be a connected component of~$G_X = G_0$, i.e.\ $W \in \Cc_0$.
  Clearly $G_X[W] \cap G_Y$ contains the graph~$\Gamma_W$, which is connected by construction.
  Secondly, any vertex shared by~$G_0$ and~$G_i$ is in~$U_i$, which is contained in~$V(\Gamma)$,
  thus any vertex shared by~$G_X$ and~$G_Y$ is in~$V(\Gamma)$.
  It follows that all vertices of $G_X[W] \cap G_Y$ belong to the subgraph $\Gamma[W] = \Gamma_W$.
  This proves that $\Gamma_W$ is a spanning subgraph of $G_X[W] \cap G_Y$, which is therefore connected.
\end{claimproof}

As $G_X=G_0$ is embeddable in $\Sbb$, Theorem \ref{thm:bn-surface} gives $\bn(G_X)=O(\log^2g)$. To conclude the proof of Theorem \ref{thm: bn-almost}, it thus remains to bound the basis number of the graph $G_Y$.
We will show that $G_Y$ has bounded treewidth, which, by Theorem \ref{thm:bn-treewidth} will allow us to conclude. In fact, it will be more convenient to bound the treewidth of the supergraph $H_Y$ of $G_Y$, which we define as follows.

For each $i\in [k]$, we let $C_i$ be a cycle with vertex set $U_i$, which cyclically connects the vertices of $U_i$ according to the cyclic order induced by the disk~$D_i$. We let $H_0\eqdef G_0\cup C_1\cup \cdots\cup C_k$.
As the disks $D_i$ have pairwise disjoint interiors, and as they only intersect $G_0$ in the vertex sets~$U_i$, we can extend the embedding of $G_0$ to an embedding of $H_0$ in $\mathbb S$ such that each~$C_i$ corresponds to the boundary~$\partial D_i$ (note that by doing so, some edges of $H_Y$ might have multiplicity more than $1$).
We now define $H_Y \eqdef G_Y\cup C_1\cup \cdots \cup C_k$,
as well as $H_\Gamma \eqdef \Gamma \cup C_1 \cup \dots \cup C_k$,
which is a subgraph of both~$H_Y$ and~$H_0$, and thus we will consider it as an embedded subgraph in $\Sbb$ (see Figure \ref{fig: Gamma}).

We stress out that while the basis number of $G \cup C_1 \cup \dots \cup C_k$ can easily be bounded using \cref{lem:bn-connected-separator}
(this is exactly the false proof presented in the introduction of this section),
this does not imply a bound on the basis number of~$G$.
The point of the first half of the proof was to replace this goal with bounding the treewidth of~$G_Y$,
for which it is of course sufficient to bound the treewidth of~$H_Y$.
We will show that~$H_\Gamma$ has bounded treewidth using \cref{coro: tw-dual}, and that a tree-decomposition of~$H_\Gamma$ combines with the path-decompositions of the vortices to prove that~$H_Y$ has bounded treewidth.
Adding the cycles~$C_i$ helps with the latter step.

\begin{claim}
 \label{clm: almost-embedd-blowup}
 We have
 $$\tw(H_Y)+1 \leq  (k+1) (\tw(H_\Gamma)+1).$$
\end{claim}

\begin{claimproof}
  Let $(T,\beta)$ be a tree-decomposition of $H_\Gamma$ of optimal width.
  Furthermore, for each vortex~$G_i$, denote by $u^i_1,\dots,u^i_{m_i}$ the vertices of~$U_i$ ordered cyclically along the boundary of~$D_i$.
  Since~$G_i$ has width at most~$k$, there is a path-decomposition of width~$k$ of~$G_i$ with bags $B^i_1,\dots,B^i_{m_i}$ of size at most~$k+1$, such that $u^i_j \in B^i_j$.

  We now define a mapping $\beta': V(T)\to 2^{V(H_Y)}$ by setting for each $t\in V(T)$
  \[ \beta'(t) \eqdef \beta(t) \cup \bigcup_{u^i_j \in \beta(t)} B^i_j. \]
  In other words, $\beta'(t)$ is obtained from~$\beta(t)$ by replacing each occurrence of a vertex $u^i_j \in U_i$ by the corresponding bag~$B^i_j$ in the decomposition of~$G_i$.
  Since~$B^i_j$ itself contains~$u^i_j$, we have $\beta'(t) \supseteq \beta(t)$.
  Also, $\beta(t)$ has size at most~$\tw(H_\Gamma)+1$, and each~$B^i_j$ has size at most~$k+1$,
  hence we have $|\beta'(t)|\leq (k+1)(\tw(H_\Gamma)+1)$ as desired.
  It only remains to prove that $(T,\beta')$ is a tree-decomposition of~$H_Y$.
  Let us check the three conditions from the definition.

  Condition \ref{item: td1} states that each vertex of~$H_Y$ is contained in some bag~$\beta'(t)$.
  For vertices of~$H_\Gamma$, this is immediate from the same condition on~$(T,\beta)$, since $\beta(t) \subseteq \beta'(t)$ for all nodes~$t$.
  For a vertex $v \in G_i$, we have $v \in B^i_j$ for some~$j$, and $u^i_j \in \beta(t)$ for some $t \in V(T)$ since~$u^i_j$ is a vertex of~$H_Y$,
  hence $v \in \beta'(t)$.
 
  Next \ref{item: td2} requires each edge to be contained in some bag.
  Once again, this is immediate for edges of~$H_\Gamma$ since $\beta(t) \subseteq \beta'(t)$,
  and for any edge~$e$ in~$G_i$, there is some bag~$B^i_j$ containing~$e$, and some $t \in V(T)$ such that $u^i_j \in \beta(t)$,
  hence $e \subseteq B^i_j \subseteq \beta'(t)$.

  It remains to prove \ref{item: td3}, i.e.\ that for each vertex $x \in V(H_Y)$, the set $T_x = \{t \in V(T) : x \in \beta'(t)\}$
  of bags containing~$x$ induces a connected subtree of~$T$.
  When~$x$ is a vertex of $V(H_\Gamma) \setminus (U_1 \cup \dots \cup U_k)$ (i.e.\ one which is not shared with any vortex~$G_i$),
  this is immediate as the bags containing~$x$ correspond to the exact same nodes in $T$, both with respect to~$\beta$ and~$\beta'$.
  Assume now $x \in V(G_i)$.
  Condition \ref{item: td3} on the path-decomposition $B^i_1,\dots,B^i_{m_i}$ of~$G_i$ gives that there is an interval~$[a,b]$ in~$[m_i]$
  such that the bags~$B^i_j$ that contain~$x$ are exactly the ones with $j \in [a,b]$.
  Let us denote by $A = \{u^i_j : j \in [a,b]\}$ the corresponding interval of vertices in~$A$.
  Due to the cycle~$C_i$, the subgraph $H_\Gamma[A]$ is connected.
  Let $T_A \eqdef T[\{t \in V(T) : \beta(t) \cap A \neq \emptyset\}]$ be the subtree of $T$ induced by the set of nodes that contain some vertex of~$A$.
  From the definition of~$\beta'$, it is clear that $T_x = T_A$, and \cref{rmk:td-connected-subgraph} applied to~$(T,\beta)$ gives that~$T_A$ is a connected subgraph of $T$.
\end{claimproof}

By Claim \ref{clm: almost-embedd-blowup}, it suffices to bound the treewidth of $H_\Gamma$ in order to conclude the proof.
This follows from \cref{coro: tw-dual} and the next claim.

\begin{claim}
 \label{clm: almost-embedd-dual}
 The dual of the graph $H_\Gamma$ has diameter at most $6k-2$.
\end{claim}
\begin{claimproof}
  First, let us point out that this claim only considers subgraphs of~$H_0$, whose embedding in~$\Sbb$ is fixed, and extends the given embedding of~$G_0$.
  There is thus no ambiguity regarding the choice of embeddings.

  Recall that~$H_\Gamma$ consists of~$\Gamma$, plus the cycles~$C_i$ corresponding to the boundary of each disk~$D_i$.
  Furthermore, $\Gamma$ was constructed so that $\partial F_i\subseteq \Gamma$ for each $i \in [k]$.
  In particular, this implies that for every $i\in [k]$ and every face~$F$ of $\Gamma$, either $F$ is disjoint from $F_i$,
  or $F\subseteq F_i$ (with in fact~$F = F_i$ in the latter case since~$\Gamma$ is a subgraph of~$G_0$).
  Finally, since~$\Gamma$ is a subgraph of~$H_\Gamma$, each face of~$H_\Gamma$ is included in some face of $\Gamma$.
  Combining these remarks, each face of~$H_\Gamma$ is either disjoint from or contained in~$F_i$ for each~$i \in [k]$.
 
 With this in mind, we now distinguish three different types of faces in $H$:
 \begin{enumerate}[label=(\arabic*)]
  \item\label{item: face1} the interiors of the $k$ disks $D_1,\ldots, D_k$ with cycle boundaries $C_1,\ldots,C_k$,
  \item\label{item: face2} the faces $F$ distinct from the disks $D_1,\ldots, D_k$, but contained in~$F_i$ for some~$i \in [k]$,
  \item\label{item: face3} and the faces $F$ whose interior is disjoint from all the $F_i$s.
 \end{enumerate}
We now show the following.
\begin{itemize}
 \item Every face of type \ref{item: face2} is adjacent in the dual~$H^{\ast}$ to a face of type \ref{item: face1}.
   Indeed, for each $i\in [k]$, if we let $D_{i_1},\ldots, D_{i_m}$ denote all disks~$D_j$ that are included in $F_i$
   (of which there is at least one), then the faces of type~\ref{item: face2} contained in~$F_i$ are exactly the arcwise-connected components of $F_i\setminus (\overline{D_{i_1}}\cup \cdots\cup \overline{D_{i_m}})$.
  Then, as $F$ is arcwise-connected, we conclude that all such components must be incident to at least one edge of some boundary $\partial D_{i_j}=C_{i_j}$.
 
 \item Every face $F$ of type \ref{item: face3} is a face of~$\Gamma$.
    Observe first that~$F$ cannot be incident to an edge of some~$C_i$.
    Indeed, since the disk~$D_i$ is contained in the interior of~$F_i$ except for the boundary points~$U_i$,
    and since~$C_i$ is embedded as the boundary of~$D_i$,
    any face incident to an edge of~$C_i$ must intersect~$F_i$.
    Now~$F$ is a face of the supergraph~$H_\Gamma$ of~$\Gamma$, which immediately implies $F \subseteq F'$ for some face~$F'$ of~$\Gamma$,
    but also its boundary~$\partial F$ is fully contained in~$\Gamma$, hence we actually have $F = F'$, as claimed.

    Then, Claim \ref{clm: union-Gamma-W} implies that $\partial F$ contains an edge incident to one of the faces $F_1,\ldots,F_k$.
    Thus~$F$ is adjacent in $H^{\ast}$ to some face of type~\ref{item: face1} or~\ref{item: face2}.
    By previous item, $F$ is then at distance at most $2$ from some face of type~\ref{item: face1} in $H^{\ast}$.
\end{itemize}
We thus proved that every face of $H$ is at distance at most $2$ in $H^{\ast}$ from some face of type \ref{item: face1}. In other words, 
the set $\mathcal D\eqdef \sg{D_1,\ldots, D_k}$ is a dominating set in the square $(H^{\ast})^2$ of $H^{\ast}$. In particular, as $H^{\ast}$ (and thus also $(H^{\ast})^2$) is connected, Remark \ref{rem: dominating-diameter} then gives $\diam(H)\leq 2\cdot\diam((H^{\ast})^2)\leq 6k-2$. 
\end{claimproof} 

The proof of Theorem \ref{thm: bn-almost} now immediately follows from all previous claims: by Claim \ref{clm: almost-embedd-dual} and Corollary \ref{coro: tw-dual}, we have
$$\tw(H_\Gamma) \leq \alpha_{\ref{coro: tw-dual}} \cdot (g+1) \cdot (6k-2), $$
where~$\alpha_{\ref{coro: tw-dual}}$ denotes the constant implicit in \cref{coro: tw-dual}.
Then by Claim \ref{clm: almost-embedd-blowup}, and because $G_Y$ is a subgraph of $H_Y$, we have
\begin{align*}
  \tw(G_Y) & \leq \tw(H_Y) \\
           & \leq (\alpha_{\ref{coro: tw-dual}} \cdot (g+1) \cdot (6k-2) + 1) (k+1) \\
           & \le \alpha_{\ref{coro: tw-dual}} \cdot (g+1) \cdot (6k^2+5k-1).
\end{align*}
Applying Theorem \ref{thm:bn-treewidth} then gives $\bn(G_Y)\leq f_{\ref{thm:bn-treewidth}}(\alpha\cdot (g+1)\cdot (6k^2+4k-2))$.
Finally, recall that since~$G_X = G_0$ is embedded in~$\Sbb$, it has basis number at most~$O(\log^2g)$ by Theorem \ref{thm:bn-surface}.
\cref{clm: almost-embedd-separation,lem:bn-almost-connected-separator} give that $\bn(G) \le \bn(G_X) + \bn(G_Y)$, and thus
$$\bn(G)= f_{\ref{thm:bn-treewidth}}(\alpha_{\ref{coro: tw-dual}} \cdot (g+1)\cdot (6k^2+5k-1)) + O(\log^2g).$$
This completes the proof of \cref{thm: bn-almost}.

\subsection{Structure theorem and proof of Theorem \ref{thm: main-minor}}
\label{sec: proof-minor}
The Graph Minor Structure Theorem states that graphs almost embeddable in a surface of bounded genus are the base bricks from which we can construct all graphs excluding a fixed minor.

\begin{theorem}[Graph Minor Structure Theorem \cite{RSXVI}]
 \label{thm: RS-structure}
 For every graph $H$, there exist some constants $a,k,g\in \mathbb N$ such that every graph $G$ excluding $H$ as a minor has a tree-decomposition of adhesion at most $k$, whose torsos all belong to $\Gc_{a,k,g}$.
\end{theorem}

Theorem \ref{thm: RS-structure} is commonly considered as one of the deepest results from graph theory, and some efforts have been spent in trying to make explicit and as good as possible the relation between the parameters $a,k,g$ and the order~$|H|$ of the excluded minor~$H$ \cite{KRW20}. Gorsky, Seweryn and Wiederrecht~\cite{GSW25} recently managed to make all these bounds polynomial, namely, they proved that the parameters $a,k$  can be chosen with value $O(|H|^{2300})$, while it follows from the original version of the Graph Minor Structure Theorem that $g$ can be chosen with value $O(|H|^2)$.

The proof of Theorem \ref{thm: main-minor} now follows from the previous results.
\begin{proof}[Proof of Theorem \ref{thm: main-minor}]
 We set $t\eqdef |H|$.
 Let $a,g,k$ and $(T,\beta)$ be given by Theorem \ref{thm: RS-structure}.
 The torsos of~$(T,\beta)$ are by assumption in~$\Gc_{a,k,g}$, which is a monotone graph class.
 By Theorem \ref{thm: bn-almost}, graphs in~$\Gc_{a,k,g}$ have basis-number at most
 \[ b'\eqdef f_{\ref{thm:bn-treewidth}}(\alpha_{\ref{coro: tw-dual}}\cdot g\cdot (6k^2+4k-2))+2a+O(\log^2g). \]
 Applying Theorem \ref{thm: main-td} thus gives us
 $$\bn(G)\leq f_{\ref{thm: main-td}}(b',k) = O((b' + k \log^2k) \cdot k^4).$$
 Using the polynomial bounds of Gorsky, Seweryn and Wiederrecht \cite{GSW25}, we can choose $a,k=O(t^{2300})$ and $g=O(t^2)$.
 Moreover, \cref{cor:bn-treewidth-polynomial} gives the bound $f_{\ref{thm:bn-treewidth}}(\ell) \le O(\ell^5)$.
 We thus obtain
 \begin{align*}
    b' & = O((gk^2 + 2a + \log^2g)^5) \\
    & = O((gk^2)^5) \\
      & = O(t^{(2 + 2 \times 2300) \times 5}) = O(t^{23010}).
 \end{align*}
 and therefore, $k \log^2k$ being small compared to~$b'$,
 \[ \bn(G) = O(b' \cdot k^4) = O(t^{23010 + 4\times 2300})=O(t^{32210}). \qedhere \]
\end{proof}

\section{Further discussion}
\label{sec: ccl}

\subsection{Graphs with unbounded basis-number}
\label{sec: Cayley}
In this short subsection, we give, using an observation of Schmeichel \cite{schmeichel1981basisnumber}, some examples of graph constructions with unbounded basis numbers and additional interesting properties. In particular, we show that using the known constructions of Ramanujan expanders, one can choose these graphs to be \emph{Cayley graphs}. We omit the definition of Cayley graphs, as we will only need here the property that every Cayley graph is regular

The next result, even though not stated explicitly, follows directly from the proof of \cite[Theorem 3]{schmeichel1981basisnumber}.

\begin{lemma}[\cite{schmeichel1981basisnumber}]
 \label{lem: bn-lowerbound}
 For any graph~$G$ with average degree $d \eqdef \frac{2|E(G)|}{|V(G)|}$ and girth~$\gamma$, we have
 \[\bn(G)\geq (1 - 2/d) \cdot \gamma.\]
\end{lemma}
As the proof of Lemma \ref{lem: bn-lowerbound} is simple and self-contained, we reproduce it here.

\begin{proof}
  Let~$n,m,k$ be the number of vertices, edges, and connected components in~$G$ respectively.
  Recall that the cycle space of $G$ has dimension $m-n+k$.
  Let $\Bb = \{C_1,\dots,C_{m-n+k}\}$ be a cycle basis of~$G$ with edge-congestion~$\bn(G)$. Then
  \[ (m-n)\gamma \leq \sum_{i=1}^{m-n+k}|E(C_i)|\leq \bn(G)\cdot m, \]
  where the first inequality holds as each~$C_i$ has length at least~$\gamma$,
  and the second inequality holds since each edge of~$G$ appears in at most~$\bn(G)$ cycles~$C_i$.
  Thus, we have
  \[ \bn(G) \ge \frac{m-n}{m} \gamma = (1 - n/m) \cdot \gamma = (1-2/d) \cdot \gamma. \qedhere \]
\end{proof}

Lemma \ref{lem: bn-lowerbound} implies in particular that any family of graphs with minimum degree~3 and unbounded girth has unbounded basis number.
Such graph classes can be obtained using the probabilist method~\cite{erdos1959probability}, or with some more explicit constructions~\cite{sachs1963girth} (see \cite[Theorem~2.13]{reiher2024girth} for a simple proof). Another way of obtaining explicit constructions, is to consider \emph{Ramanujan expanders}, which moreover are Cayley graphs, and thus enjoy stronger algebraic properties.
The first known constructions of Ramanujan expanders are due to Lubotzky, Phillips, Sarnak and independently Margulis \cite{Lubotzky, Margulis}.
In particular, a characteristic of these constructions is that they produce graphs with girth $\Omega(\log(n))$. Next Theorem sums up some properties of such families, enabling to apply Lemma \ref{lem: bn-lowerbound}. The constructions from next theorem were provided by Morgenstern \cite{Morgenstern}, and generalise the aforementioned ones to obtain infinite families of Ramanujan graphs which are $(q+1)$-regular for any arbitrary prime power $q$ (see also \cite{Chiu92} for the case $q=2$).

\begin{theorem}[\cite{Morgenstern}]
 \label{thm: expanders}
  For every odd prime power $q$, there exists 
  an infinite subset $I_q\subseteq \mathbb N$ such that for each $d\in I_q$, 
  there exists a $(q+1)$-regular Cayley graph $G_{q,d}$ with order at least $\tfrac{q^{3d}-q^d}{2}$, and girth at least $\tfrac{2}{3}\log_q(|V(G_{q,d})|)+1$.
\end{theorem}

Combining Lemma \ref{lem: bn-lowerbound} with Theorem \ref{thm: expanders}, we immediately obtain the following.

\begin{proposition}\label{prop: expanders}
 There exists families $\Gc_1,\Gc_2$ of Cayley graphs with unbounded basis number and girth, such that
 \begin{enumerate}
  \item $\Gc_1$ contains only cubic graphs, 
  \item $\Gc_2$ contains graphs of arbitrarily large degree, and thus also of arbitrarily large connectivity.
 \end{enumerate}
\end{proposition}

\begin{proof}
 We let $q$ be a prime power, and let $\Hc_q$ denotes the infinite family $\sg{G_{q,d}: d\in I_q}$ given by Theorem \ref{thm: expanders}. We check that for each $q$, graphs in $\Hc_q$ have unbounded basis number. Fix $q$ a prime power and let $d\in I_q$. We let $n_{q,d}\eqdef |V(G_{q,d})|$. As $G_{q,d}$ is $(q+1)$-regular, we have $\Delta(G_{q,d})\cdot n_{q,d}=(q+1)\cdot n_{q,d}=2|E(G_{q,d})|$. In particular, Lemma \ref{lem: bn-lowerbound} then gives
 \begin{align*}
  \bn(G_{q,d})&\geq \frac{(q+1)n_{q,d}-2n_{q,d}+2}{(q+1)n_{q,d}}\gamma(G_{q,d})\\
  &\geq \left(q-1+\frac{2}{(q+1)n_{q,d}}\right)\left(\tfrac{2}{3}(\log_q(q^{3d}-q^d)-\log_q(2))+1\right).
 \end{align*}
 Thus, when $d\in I_q$ tends to infinity, we also obtain that $\bn(G_{q,d})$ tends to infinity. 
 
 To conclude the proof, we set $\Gc_1\eqdef \Hc_2$ and $\Gc_2\eqdef \bigcup_{q}\Hc_q$, where the union is taken over all prime powers $q$. Clearly, $\Gc_1$ and $\Gc_2$ have unbounded basis number. Moreover, a result of Godsil and Royle \cite[Theorem 3.4.2]{GR01} states that for every $d\geq 0$, every finite vertex-transitive $d$-regular is $\lceil\tfrac23(d+1)\rceil$-connected. This implies in particular that $\Gc_2$ contains graphs of arbitrary large connectivity.
\end{proof}

\subsection{Future directions}
\label{sec: questions}

\paragraph{Bounds}
A first obvious question left open by our work concerns the optimality of the bound from Theorem \ref{thm: main-minor}. Clearly, it is quite unlikely that $O(t^{32210})$ is optimal.
In particular, a proof that $H$-minor-free graphs have bounded basis number without relying on the Graph Minor Structure Theorem
would be very interesting in its own and could lead to improvements of our bound.
Note also that to our knowledge, it is not excluded that a bound in $\log(t)^{O(1)}$ could hold. 

In terms of lower bounds, we note that for every $t\in \mathbb N$, there exists a graph with basis number $\Omega(\log(t))$ that excludes $K_t$ as a minor. This can be seen using for example \cite[Proposition 2.2]{DR05}, which states that there exists a constant $c\in \mathbb R$ such that for infinitely many $t\in \mathbb N$, there exists a cubic graph $G_t$ with girth at least $\tfrac83\log(t) - c$ that excludes $K_t$ as a minor. In particular, \Cref{lem: bn-lowerbound} implies that such graphs have basis number $\Omega(\log(t))$. Note that we cannot hope for any better lower bound using this argument, as the authors from \cite{DR05} showed the existence of a constant $c\in \mathbb R$ such that such that for all $t\in \mathbb N$, every graph with girth at least $6\log(t)+3\log\log(t) + c$ and minimum degree at least $3$ admits $K_t$ as a minor. 

\paragraph{Hereditary classes}
As pointed out in \cref{cor:monotone-bn}, the property of being proper minor closed characterises basis number boundedness for monotone graph classes.
To go beyond this, it seems natural to consider graph classes that are \emph{hereditary}, that is, closed only under taking induced subgraphs ---
or equivalently closed under taking \emph{induced minors}, as edge contractions do not increase the basis number.
A full characterisation of induced-minor closed classes with bounded basis number seems well beyond reach,
but we believe that it would be interesting to find examples or establish properties of such classes.
For instance, generalising \cref{thm:bn-treewidth}, one may ask whether basis number can be upper-bounded by a function of other classical graph parameters such as \emph{clique-width}.
Let us point out that such a bound cannot exist for the more general parameter \emph{twin-width},
as subdividing all edges sufficiently many times decreases the twin-width of any graph to~4~\cite{berge2022tww4} without affecting its basis number.

A question that occurred to us during this work is whether in Theorem \ref{thm: main-td}, one can weaken the hypothesis that the family $\Gc$ of graphs is monotone, by only assuming instead $\Gc$ to be hereditary. This does not seem out of reach, and would allow to consider significantly more general classes of graphs constructed through tree-decompositions, as for example the cliques and the complete bipartite graphs are known to have basis numbers at most $4$ (see \cite{schmeichel1981basisnumber}).

\paragraph{Infinite graphs}
Finally, cycle bases and basis number can be generalised to infinite graphs. In fact, two different definitions have been considered in the litterature. The first one is more restrictive and defines the cycle space of a graph as its set of finite even subgraphs, equipped with $\oplus$. With respect to this notion of cycle space, Thomassen \cite[Theorem 7.4]{Thomassen80} generalised MacLane's planarity criterion, by proving that the infinite graphs having basis number at most~$2$ are exactly those graphs admitting a \emph{vertex accumulation free} planar embedding, that is, an embedding in $\mathbb R^2$ such that every point of $\mathbb R^2$ has an open neighbourhood containing only finitely many vertices. A natural question is whether planar graphs, and more 
generally graphs embedded in a fixed surface of bounded genus, bounded treewidth graphs, or graphs excluding a fixed minor have bounded (or even finite) basis number. Nathan Bowler (private communication) showed us a proof that with respect to this notion, infinite planar graphs have basis number at most $4$. 
As a side remark, observe that every Cayley graph of a \emph{finitely presented group} has finite basis number. In particular, as planar and more generally minor-excluded finitely generated groups are finitely presented \cite{Droms, EGL23}, minor-excluded Cayley graphs (and more generally locally finite quasi-transitive graphs) have finite basis number. 

A second possible definition of cycle space, called \emph{topological cycle space} consists in considering all subgraphs generated by \emph{infinite cycles}. We refer to \cite[Section 8.7]{DiestelBook} and \cite{DiestelCycles} for a comprehensive exposition on the topic, and to \cite{BS06} for a proof of a generalisation of MacLane's planarity criterion in this case. Again, basis number can be defined indentically with respect to the topological cycle space, and it makes sense to ask whether all aforementioned classes of graphs also have bounded basis number in this context.

\section*{Acknowledgements}
We are deeply grateful to Babak Miraftab, Pat Morin and Yelena Yuditsky for very instructive exchanges, and 
for sharing with us their results once we learned that part of our respective works overlapped. Moreover, we thank Pat Morin for pointing to us a bound for Lemma \ref{lem:bn-small-separator} simplifying our original bound.
We also Louis Esperet for some useful discussion, that lead in particular to a simplification of our initial proof of Corollary \ref{coro: tw-dual}, and Nathan Bowler for showing us a proof that infinite planar graphs have basis number at most $4$.

\bibliographystyle{plainurl}
\bibliography{refs}
\end{document}